\documentclass[a4paper,10pt]{amsart}
\usepackage[english]{babel}
\usepackage[centering]{geometry}

\usepackage[backend=biber, citestyle = numeric-comp, bibstyle=nature,doi=false,isbn=false,url=false,eprint=false,date=year, giveninits=true]{biblatex}
 \renewbibmacro{in:}{}
\DeclareFieldFormat{pages}{#1}
\DeclareFieldFormat[article,periodical]{volume}{\mkbibbold{#1}}
\DeclareFieldFormat[article,periodical]{title}{\MakeLowercase{\mkbibemph{#1}}}
\DeclareFieldFormat[article,periodical]{journaltitle}{#1}
\DeclareFieldFormat[inproceedings]{booktitle}{#1}

\usepackage{amsmath,amssymb,amsthm}

\usepackage{tikz-cd}

\usepackage{hyperref}
\usepackage[noabbrev,capitalise]{cleveref}

\newtheorem{theorem}{Theorem}[section] 
\newtheorem{lemma}[theorem]{Lemma}
\newtheorem{proposition}[theorem]{Proposition}
\newtheorem{conjecture}[theorem]{Conjecture}

\theoremstyle{definition}

\newtheorem{corollary}[theorem]{Corollary}
\newtheorem{remark}{Remark}

\numberwithin{equation}{section}

\let\svqty\qty
\usepackage{physics}
\let\qty\svqty

\usepackage[shortlabels]{enumitem}
\usepackage[numbered]{bookmark}

\usepackage{float}
\usepackage{graphicx,xcolor}
\usepackage{grffile}
\usepackage[labelfont=bf, labelsep=space]{caption} % , format=hang] lijkt niet te werken 
\usepackage[labelfont=default]{subcaption}

\usepackage{dsfont} % for identity: \mathds{1}

\usepackage{upgreek} % for pi for hopf fibration

\newcommand{\hpi}{\uppi} % pi for the hopf fibration
\newcommand{\ttau}{\uptau} % tau for the twistor fibration

\usepackage{csquotes}
\DeclareQuoteAlias[american]{english}{dutch}
\usepackage{todonotes}

\usepackage{arydshln} % table with dashed vertical line
\usepackage{multirow}

\newcommand*{\reals}{\mathbb R}
\newcommand*{\complexs}{\mathbb{C}}

\newcommand*{\set}[1]{\l\{ #1 \r\}} 
\renewcommand{\l}{\left}
\renewcommand{\r}{\right}

\renewcommand*{\epsilon}{\varepsilon}

\renewcommand*{\phi}{\varphi}

\renewcommand*{\tilde}{\widetilde}

\newcommand{\CP}{{\complexs P}}
\newcommand{\sphere}[1]{\mathbb{S}^{#1}}
\newcommand{\quaternions}{\mathbb{H}}

\DeclareMathOperator{\diag}{diag}

\newcommand{\lift}[1]{\widetilde{#1}}

\newcommand{\Jo}{J_\circ}

\newcommand{\Dtwo}{\mathcal{D}_1^2}
\newcommand{\Dfour}{\mathcal{D}_2^4}

\newcommand{\FSm}{g_{\circ}}%\hat{g}}
\newcommand{\FSnabla}{\nabla^{\circ}}%\widehat{\nabla}}

\DeclareMathOperator{\SU}{SU}
\DeclareMathOperator{\U}{U}
\DeclareMathOperator{\Sp}{Sp}

\newcommand\su{\mathfrak{su}}
\newcommand\so{\mathfrak{so}}
\renewcommand\u{\mathfrak{u}}
\renewcommand\sp{\mathfrak{sp}}
\newcommand\m{\mathfrak{m}}
\newcommand\h{\mathfrak{h}}
\DeclareMathOperator{\Ad}{Ad}
\DeclareMathOperator{\Lie}{Lie}

\DeclareMathOperator{\Span}{Span}

\title[Classification results for totally real surfaces of nearly Kähler $\CP^3$]{Classification results for totally real surfaces of nearly Kähler $\CP^3$}

\author{Michaël Liefsoens \and Hui Ma \and Luc Vrancken}

\address{M. Liefsoens, KU\ Leuven, Department of Mathematics, Celestij\-nenlaan 200B -- Box 2400, 3001 Leuven, Belgium}
\email{michael.liefsoens@kuleuven.be}
\address{H.~Ma, Department of Mathematical Sciences, Tsinghua University, Beijing 100084, P. R. China}
\email{ma-h@tsinghua.edu.cn}
\address{L.~Vrancken, Ceramaths,
Universit\'e Polytechnique Hauts-de-France, F-59313 Valenciennes, France;
Ceramaths, INSA Hauts de France, F 59313 Valenciennes, France,
Department of Mathematics, KU Leuven, Celestijnenlaan 200B, Box 2400, BE-3001 Leuven, Belgium.}
\email{luc.vrancken@uphf.fr}

\thanks{M. Liefsoens is supported by the Research Foundation Flanders (FWO) with project 11PG324N. H.~Ma is partially supported by the National Natural Science Foundation of China (Grant No. 12471048).}


\begin{document}

\begin{abstract}
	Totally real surfaces in the nearly Kähler $\CP^3$ are investigated and are completely classified under various additional assumptions, resulting in multiple new examples. Among others, the classification includes totally real surfaces that are extrinsically homogeneous; or minimal; or totally umbilical; or Codazzi-like (including parallel and non-parallel examples). \vspace{-.5cm}
\end{abstract}

\maketitle

%%%%%%%%%%%%%%%%%%%%%%
\section{Introduction}
%%%%%%%%%%%%%%%%%%%%%%

In 1959, Tachibana \cite{tachibana_almost-analytic_1959} coined the definition of a \emph{nearly Kähler} manifold: a Riemannian manifold $(M,g)$ with a compatible almost complex structure $J$ such that the covariant derivative of $J$ is skew-symmetric. If $\nabla J$ is moreover non-degenerate, $(M, g, J)$ is called a \emph{strict nearly Kähler} manifold. These manifolds are a natural and interesting generalisation of Kähler manifolds, and also have applications in physics, see \cite{atiyah2002}. In the 1960s, the concept of a nearly Kähler manifold was independently introduced and studied by Gray \cite{gray_nearly_1970,gray_examples_1966,gray_structure_1976}. He also chose the name of this class of manifolds.

There are several results that highlight the central role of \emph{six}-dimensional nearly Kähler manifolds in the general theory of nearly Kähler manifolds. Firstly, six-dimensional nearly Kähler manifolds are the lowest-dimensional non-Kähler examples. Secondly, Nagy's structure theorem \cite{nagy_nearly_2002} states that any strict, complete, and simply connected nearly Kähler manifold can be decomposed into a Riemannian product with factors of six-dimensional nearly Kähler manifolds, homogeneous nearly Kähler spaces, or twistor spaces over quaternionic manifolds with positive scalar curvature.

Butruille \cite{butruille_classification_2005} classified all homogeneous six-dimensional strictly nearly Kähler manifolds, showing that there are only four such manifolds: the round six-sphere $\sphere{6}$, the complex projective space $\CP^3$ (not with its Fubini-Study metric), the product manifold $\sphere{3}\times\sphere{3}$ (not with its product metric), and the flag manifold $F_{1,2}$, which consists of full flags in $\complexs^3$.
In 2017, Foscolo and Haskins \cite{foscolo_new_2017} made a significant breakthrough by constructing the first examples of complete inhomogeneous six-dimensional nearly Kähler manifolds, with as differentiable manifolds $\sphere{6}$ and $\sphere{3}\times\sphere{3}$. 
Notably, nearly Kähler $\CP^3$ and the flag manifold are unique in that they belong to all three categories in Nagy's classification.
However, among these, nearly Kähler $\CP^3$ and its submanifold theory remain the least studied.

In order to understand these spaces it is important to study their submanifolds, and in particular those that are special with respect to the nearly Kähler structure. 
There are two natural classes of such submanifolds which we can consider, namely, almost complex submanifolds, whose tangent space is preserved by $J$, and totally real submanifolds, whose tangent space is mapped to the normal space by $J$.
Almost complex submanifolds of six-dimensional homogeneous strict nearly Kähler spaces necessarily are two-dimensional and have been widely researched: in the nearly Kähler $\sphere{6}$ by Bryant \cite{bryant1982}, in $\sphere{3}\times\sphere{3}$ by Dioos \cite{dioos2018a}, in $\CP^3$ by Aslan \cite{aslan2022a}, and finally in $F_{1,2}$ by Cwiklinski and the third author \cite{cwiklinski2022}.
In contrast, totally real submanifolds in these ambient spaces (which have dimension at most three) are less studied, in particular those of dimension two. In the sphere $\sphere{6}$, there are some results: in \cite{bolton1994,bolton1997}, the authors classified totally real minimal surfaces.
Totally real 3-dimensional manifolds are called Lagrangian and must automatically be minimal. These are well studied in (among others) \cite{enoyoshi_lagrangian_2020,hu_ricci_2020,hu_rigidity_2019} for $\sphere{6}$, \cite{bektas_lagrangian_2019,dioos_lagrangian_2018} for the nearly Kähler $\sphere{3}\times\sphere{3}$, \cite{aslan_special_2023,liefsoens2024,storm_note_2020} for the nearly Kähler $\CP^3$, and \cite{storm_note_2020} for $F_{1,2}$. 

In this paper, we initiate the study of totally real surfaces of the nearly Kähler $\CP^3$. In \Cref{sec:description_global}, we introduce the nearly Kähler structure on $\CP^3$, which also involves the standard Kähler structure $\Jo$. In \Cref{sec:frame}, we use the almost product structure to distinguish different types of totally real surfaces and introduce a special frame to study them. In \Cref{sec:examples}, we give two main types of examples. The first main class of examples consists of several subclasses of flat tori, as well as one immersion of a sphere. We will show that these are the only extrinsically homogeneous totally real surfaces, see \Cref{thm:EH}. The second main class of examples are products of a Legendre curve in $\sphere{3}$ with (its generalisation  for the quaternions of) a Legendre curve in $\sphere{7}$ (as introduced in \Cref{sec:legendre_S7}). These are all the surfaces for which the covariant derivative of the second fundamental form is a Codazzi tensor, see \Cref{thm:classification_CL_TR}. This class includes all the parallel examples (those with $\nabla h = 0$), and they are classified completely in \Cref{thm:classification_parallel_1}. Finally, we characterise various elements in these classes by properties of either $\Jo$ or of the second fundamental form. In particular, we classify the totally real surfaces that are additionally almost complex for Kähler $\CP^3$ in \Cref{thm:classification_AC_KCP3}, or totally umbilical in \Cref{thm:classification_totally_umbilical}, or totally geodesic for Kähler $\CP^3$ in \Cref{thm:tg_in_KCP3}. Finally, we obtain a characterisation of the flat minimal Lagrangian immersion in the Kähler $\CP^2 \subset \CP^3$ and the Clifford torus in $\reals P^3 \subset \CP^3$.

%%%%%%%%%%%%%%%%%
\section{Description of the nearly Kähler \texorpdfstring{$\CP^3$}{ℂP³}}\label{sec:description_global}
%%%%%%%%%%%%%%%%%
One way of describing nearly Kähler $\CP^3$ is as a twistor space over the quaternionic $\mathbb{H}P^1 \cong \sphere{4}$. Being the domain of a fibration, this description does not allow for an easy description of submanifolds. We recall two other descriptions. One as a Riemannian homogeneous space, which is sometimes very useful, and one as the base space of the Hopf fibration coming from $\sphere{7}$, which is the main tool we use.

%%%%%%%%%%%%%%%%%
\subsection{Nearly Kähler \texorpdfstring{$\CP^3$}{ℂP³} via the Hopf fibration}\label{sec:description}
%%%%%%%%%%%%%%%%%

Recall the Hopf fibration $\hpi: \sphere{7} \to \CP^3$, which is a Riemannian submersion when $\sphere{7}$ is equipped with its round metric and $\CP^3$ with its Fubini-Study metric $\FSm$. There is an almost complex structure $\Jo$ on $\CP^3$ inherited from multiplication by $i$ in $\complexs^4$. Moreover, recall that $(\CP^3, \FSm)$ is Einstein, has constant holomorphic sectional curvature $4$, and the identity component of its isometry group is \[\set{ \CP^3 \to \CP^3: \hpi(x) \mapsto \hpi(U x) \mid U \in \SU(4) }.\]

Nearly Kähler $\CP^3$ can be described as the twistor space over $\sphere{4} \cong \quaternions P^1$, where $\quaternions$ denotes the quaternions. 
We denote the twistor fibration with $\ttau$. In \Cite{liefsoens2024}, a more concrete description is given of nearly Kähler $\CP^3$, and of $\CP^3$ with any homogeneous metric. 
There, the seven-sphere is embedded in $\quaternions^2$, and the following distributions are considered: $V(p) = ip$, $\lift{\Dtwo}(p) = \Span\set{ j p, k p }$ and $\lift{\Dfour} = (V \oplus \lift{\Dtwo})^{\perp} \subset T\sphere{7}$. 
It is proven that $\Dtwo = \dd \hpi (\lift{\Dtwo})$ and $\Dfour = \dd \hpi (\lift{\Dfour})$ are well defined and orthogonal for the Fubini-Study metric on $\CP^3$. 
Moreover, $\Dtwo$ is integrable and its leafs are the fibers of the twistor fibration. 

An almost product structure $P$ was defined in \Cite{liefsoens2024} by
\begin{equation}
    PA = - A \qquad PX = X \qquad A \in \Dtwo, X \in \Dfour.
\end{equation}
From this, a new almost complex structure is defined as $J = P \Jo = \Jo P$. By taking the metric 
\[ g(X,Y) = \frac{3}{2} \FSm(X,Y)+ \frac{1}{2} \FSm(X,P Y), \] the space $(\CP^3, g, J)$ is a nearly Kähler manifold. 

The Levi-Civita connections $\FSnabla$ and $\nabla$ of $\FSm$ and $g$ are linked as $\nabla_X Y = \FSnabla_X Y + D(X,Y)$, where $D$ is given by $D(X,Y) = -\frac{1}{2} \frac{J + \Jo}{2} G( P X, Y)$.
Moreover, the isometries of nearly Kähler $\CP^3$ have been characterised in \Cite{liefsoens2024} as those Kähler isometries that commute with $P$. Finally, the following expression for the Riemann curvature was given for nearly Kähler $\CP^3$:
\begin{align*}
    R(X,Y)Z 
    &= (X \wedge Y)Z
    \\
    &\quad
    + \frac{1}{2}\bigg( (X \wedge Y)Z + (\Jo X \wedge \Jo Y)Z + 2g(X,\Jo  Y) \Jo  Z \bigg) \\
    &\quad
    - \frac{1}{4}\bigg( (X \wedge Y)Z  + (J X \wedge J Y)Z + 2 g(X,JY)JZ \bigg) \\
    &\quad
    - \frac{1}{2}\bigg( (X \wedge Y)P Z + P (X \wedge Y)Z  - 2 P( X \wedge Y)P Z \bigg)  .
\end{align*}
Here, $(X \wedge Y)Z = g(Y,Z)X-g(X,Z)Y$.

Finally, we recall some useful equations that hold for the different structures on nearly Kähler $\CP^3$. First, we have that nearly Kähler is of constant type, with the constant given by unity with the choices made in this work. We then have
\begin{subequations}
\begin{align}
    \norm{G(X,Y)}^2 &= \norm{X}^2\norm{Y}^2 - g(X,Y)^2 - g(X, JY)^2 \\ 
    \nabla G(X, Y,Z) &= J (Y\wedge Z)X  - g(JY,Z)X  \label{eq:nablaG_curv} \\
    G(X,G(Y,Z)) &= (Y\wedge Z)X  + J (Y\wedge Z)J X \label{eq:GG_curv}  \\ 
    g(G(X,Y),G(Z,W)) &= g\l( (Z\wedge W)Y, X\r) + g\l( J(Z\wedge W)JY, X\r),
 \end{align} 
 \end{subequations}
  with $G= \nabla J$.
Second, we have the covariant derivatives of the almost product structure $P$ and of the Kähler almost complex structure $\Jo$:
\begin{align}\label{eq:derivatives_P_J0}
    (\nabla P)(X,Y) &= P G\l( \frac{J + \Jo}{2} X, Y \r), &
    (\nabla \Jo)(X,Y) &= \frac{J + \Jo}{2} G\l( \frac{J-\Jo}{2} X, Y \r).
\end{align}

The main tool in this work is the use of the fundamental equations of Riemannian submanifolds, together with the second equation of \Cref{eq:derivatives_P_J0}. In all examples thus far encountered within the nearly Kähler $\CP^3$, if the data satisfies the fundamental equations together with this additional requirement, existence and uniqueness can be proven. However, a general existence and uniqueness result has not been proven, and so these are ad hoc arguments. Given the experience that all submanifolds of the nearly Kähler $\CP^3$ seem to satisfy such a statement, we make the following conjecture. 
\begin{conjecture}
	\underline{Existence.} Let $(M,g)$ be an $n$-dimensional simply connected Riemannian manifold $(n<6)$ and let $\mathcal{E}$ be a Riemannian vector bundle of rank $p<6$ over $M$ with metric compatible connection $\nabla^\mathcal{E}$ and curvature tensor $R^\mathcal{E}$. Suppose $h^\mathcal{E}$ is a smooth and symmetric map from $TM\times TM$ to $\mathcal{E}$ and define for each section $\xi$ of $\mathcal{E}$ the map $A_\xi^\mathcal{E}$ by $g(A_\xi^\mathcal{E}X,Y)=g(h^\mathcal{E}(X,Y),\xi)$. Finally, assume that $(\nabla^\mathcal{E},h^\mathcal{E},A^\mathcal{E},R^\mathcal{E})$ satisfy the Gauss, Codazzi and Ricci equations with as ambient manifold the nearly Kähler $\CP^3$; and that the equation $(\nabla \Jo)(X,Y) = \frac{J + \Jo}{2} G\l( \frac{J-\Jo}{2} X, Y \r)$ is satisfied.
	Then, there exists an isometric immersion $f:M\to (\CP^3, g, J)$ and a vector bundle isometry $\phi:\mathcal{E}\to T^\perp M$ such that $\nabla^\perp \phi = \phi \nabla^\mathcal{E}$ and $h=\phi \circ h^\mathcal{E}$.
    \textbf{}
	\\
    \underline{Uniqueness.} Let $f_1,f_2:M\to (\CP^3,g, J)$ be isometric immersions of a Riemannian manifold into the nearly Kähler $\CP^3$. Assume that there exists a vector bundle isometry $\phi:T^{\perp,f_1}M\to T^{\perp,f_2}M$ such that $\phi \circ \nabla^{\perp,f_1} = \nabla^{\perp,f_2} \circ \phi$ and $\phi\circ h^{f_1} = h^{f_2}$.
	Then there exists an isometry $\psi:\tilde{M}_c\to\tilde{M}_c$ such that $\psi \circ f = g$ and $\eval{\dd \psi}_{T{^{\perp,f}}M} = \phi$.
\end{conjecture}

%%%%%%%%%%%%%%%%%
\subsection{Nearly Kähler \texorpdfstring{$\CP^3$}{ℂP³} as a homogeneous space}\label{sec:homogeneous_description}
%%%%%%%%%%%%%%%%%

The nearly Kähler $\CP^3$ can be written as a Riemannian homogeneous space $\CP^3 = \frac{\Sp(2)}{\SU(2) \times \U(1)}$, where we see $\Sp(2)$ as $\U(4) \cap \Sp(4, \complexs)$. The metric on $\sp(2)$ is (minus) the Killing form: $g(x,y) = \frac{1}{4} \tr(x^\dagger y)$, where $x^\dagger = \overline{x}^T$ denotes the conjugate transpose of $x$. Take $\set{h_0, h_1, h_2, h_3}$ as generators of the isotropy $\u(1) \oplus \su(2)$, with 
\begin{align*}
    & & h_0 &= \sqrt{2} i \diag(1, 0, -1, 0),
                                % \mqty( 1 & 0 & 0 & 0 \\
                                % 0 & 0 & 0 & 0 \\
                                % 0 & 0 & -1 & 0 \\
                                % 0 & 0 & 0 & 0 ), 
     & & 
    \\ h_1 &= \sqrt{2} i \mqty( 0 & 0 & 0 & 0 \\
                                0 & 1 & 0 & 0 \\
                                0 & 0 & 0 & 0 \\
                                0 & 0 & 0 & -1 ), 
    & h_2 &= -\sqrt{2} i \mqty( 0 & 0 & 0 & 0 \\
                                0 & 0 & 0 & 1 \\
                                0 & 0 & 0 & 0 \\
                                0 & 1 & 0 & 0 ), 
    & h_3 &= \sqrt{2} \mqty( 0 & 0 & 0 & 0 \\
                                0 & 0 & 0 & 1 \\
                                0 & 0 & 0 & 0 \\
                                0 & -1 & 0 & 0 ), 
\end{align*}
and let $\mathfrak{m}$ be spanned by $\set{m_1,\hdots, m_6}$, given by 
\begin{align*}
    m_1 &= \sqrt{2} \mqty( 0 & 0 & 1 & 0 \\
                                0 & 0 & 0 & 0 \\
                                -1 & 0 & 0 & 0 \\
                                0 & 0 & 0 & 0 ), 
    & m_2 &= -\sqrt{2} i \mqty( 0 & 0 & 1 & 0 \\
                                0 & 0 & 0 & 0 \\
                                1 & 0 & 0 & 0 \\
                                0 & 0 & 0 & 0), 
    & m_3 &=  \mqty( 0 & 0 & 0 & 1 \\
                                0 & 0 & 1 & 0 \\
                                0 & -1 & 0 & 0 \\
                                -1 & 0 & 0 & 0 ), 
    \\
    m_4 &= \mqty( 0 & 1 & 0 & 0 \\
                                -1 & 0 & 0 & 0 \\
                                0 & 0 & 0 & 1 \\
                                0 & 0 & -1 & 0 ), 
    & m_5 &= i \mqty( 0 & 0 & 0 & 1 \\
                                0 & 0 & 1 & 0 \\
                                0 & 1 & 0 & 0 \\
                                1 & 0 & 0 & 0 ), 
    & m_6 &= i \mqty( 0 & 1 & 0 & 0 \\
                                1 & 0 & 0 & 0 \\
                                0 & 0 & 0 & -1 \\
                                0 & 0 & -1 & 0). 
\end{align*}
This gives a reductive decomposition $\sp(2) = (\u(1)\oplus\su(2)) \oplus \mathfrak{m}$. With these choices, the nearly Kähler almost complex structure $J$ acts as 
\begin{align*}
    J m_1 &= m_2, &
    J m_2 &= -m_1, &
    J m_3 &= -m_5, &
    J m_4 &= -m_6, &
    J m_5 &= m_3, &
    J m_6 &= m_4,
\end{align*}
while the almost product structure $P$ acts as 
\begin{align*}
    P m_i &= -m_i, &
    P m_j &= m_j, & i \in \set{1,2}, j \in \set{3,4,5,6}.
\end{align*}

For the purposes of this paper, this is all we need of this description of the nearly Kähler $\CP^3$.

%%%%%%%%%%%%%%%%%
\section{A frame for totally real surfaces}\label{sec:frame}
%%%%%%%%%%%%%%%%%

Let $M$ be a totally real surface of the nearly Kähler $\CP^3$. Suppose $U, V$ form an orthonormal frame on $M$. Since $M$ is totally real, $g(U, JV) = 0$. We define the following frame for $T \CP^3$ on some open of $M$:
\begin{align} \label{eq:frame_TR_surface}
    e_1 &= U, &
    e_2 &= V, &
    e_3 &= JU, &
    e_4 &= JV,  &
    e_5 &= G(U,V), &
    e_6 &= JG(U,V) .
\end{align}
 From \Cref{eq:GG_curv} and up to symmetry $G(X,Y) + G(Y,X) = 0$, we then find $G(e_1, e_3) = G(e_2, e_4) = G(e_5, e_6) = 0$ and 
\begin{align*}
    G(e_2, e_5) &= - G(e_4, e_6) = e_1 &
    G(e_1, e_5) &= - G(e_3, e_6) = - e_2 &
    G(e_2, e_6) &= G(e_4, e_5) = -e_3 \\ 
    G(e_1, e_6) &= G(e_3, e_5) = e_4 &
    G(e_1, e_2) &= - G(e_3, e_4) = e_5  &
    G(e_1, e_4) &= - G(e_2, e_3) = -e_6
\end{align*}

Given the frame of \Cref{eq:frame_TR_surface}, we can introduce functions $b_i, c_i$ to express the almost complex structure $\Jo$ in this frame as $\Jo U = + b_1 V + b_2 JU + b_3 JV + b_4 G(U,V) +b_5 JG(U,V)$ and $\Jo V = - b_1 U + b_3 JU + c_3 JV + c_4 G(U,V) + c_5 JG(U,V)$.

Since $P=-J\Jo$ restricted to the tangent space turns out to be a symmetric operator, we can diagonalise by performing a rotation in $U,V$, such that $b_3 = 0$. We obtain
\begin{equation}\label{eq:J0def_general}
    \begin{cases}
    \Jo U &= b_1 V + b_2 JU + b_4 G(U,V) +b_5 JG(U,V), \\
    \Jo V &= - b_1 U + c_3 JV + c_4 G(U,V) + c_5 JG(U,V).
    \end{cases}
\end{equation}
Note that then $\FSm(U,V) = - \frac{1}{4} g( U, P V) = - \frac{1}{4} g( J U, \Jo V) = 0$. 

The above frame is not unique. We are still allowed to change the signs of $U$ and/or $V$; and to interchange $U$ and $V$. Doing so, if necessary by restricting to an open dense subset, we may assume that $b_1 \ge 0$. For the same reason, we may also assume that either $b_4^2+b_5^2 \ne 0$ or $b_4=b_5=c_4=c_5=0$. 

We introduce the local functions $d_1, d_2$ such that the intrinsic connection $\nabla^M$ takes the following form:
\begin{align*}
    \nabla^M_U U &= d_1 V, &
    \nabla^M_U V &= -d_1 U, &
    \nabla^M_V U &= -d_2 V, &
    \nabla^M_V V &= d_2 U.
\end{align*}
Let $h$ denote the second fundamental form. From the Gauss formula, we find
\begin{align*}
    \nabla_U e_1 &= d_1 V + h(U,U), & \nabla_V e_1 &= -d_2 V + h(U,V), &
    \nabla_U e_2 &= -d_1 U + h(U,V), \\ \nabla_V e_2 &= d_2 U + h(V,V), &
    \nabla_U e_3 &= J \nabla_U e_1, & \nabla_V e_3 &= G(V,U) + J \nabla_V e_1, \\
    \nabla_U e_4 &= G(U,V) + J \nabla_U e_2, & \nabla_V e_4 &= d_2 JU + Jh(V,V),
\end{align*}
and from \Cref{eq:nablaG_curv}, we have 
\begin{align*}
    \nabla_U e_5 &= - JV + G( \nabla_U e_1,V) + G( U,\nabla_U e_2), & \nabla_V e_5 &= JU + G( \nabla_V e_1,V) + G( U,\nabla_V e_2), \\
    \nabla_U e_6 &= G(U, e_5) + J \nabla_U e_5,  & \nabla_V e_6 &= G(V, e_5) + J \nabla_V e_5.
\end{align*}

Given the frame of \Cref{eq:frame_TR_surface}, we can easily show the following fact, which is similar to the situation for Lagrangian submanifolds.
\begin{lemma}
    Let $h$ denote the second fundamental form of a totally real surface in a nearly Kähler manifold $(M, g, J)$. Then, $(X,Y,Z) \mapsto g(h(X,Z), JY )$ is totally symmetric for all $X,Y,Z$ tangent to $M$.
\end{lemma}
\begin{proof}
    We only have to prove that $g(h(X,Z), JY ) = g( h(Y,Z), J X)$ for all $X,Y,Z$ tangent to $M$. For this, we start from the fact that $g(X, JY) = 0$, since $M$ is totally real. Differentiating this equation with respect to $Z$ and using metric compatibility, we obtain $g(h(X,Z), JY ) + g( X, G(Z, Y) ) + g( X,J  \nabla_Z Y) = 0$. Since $G$ applied to two tangent vectors is normal to $M$, and by using metric compatibility of $J$, we find $g(h(X,Z), JY ) = g( h(Y,Z), J X)$. 
\end{proof}

From the above result, we can introduce functions $k_i, m_i, k, m$ ($i=1,2,3,4$) such that $h(U,U) = k_1 e_3 + k_2 e_4 + k_3 e_5 + k_4 e_6$, $h(V,V) = m_1 e_3 + m_2 e_4 + m_3 e_5 + m_4 e_6$
and $h(U,V) = k_2 e_3 + m_1 e_4 + k e_5 + m e_6$.

We now use the fact that $\Jo$ is a metric-compatible almost complex structure, in order to divide the totally real surfaces into three different types. 

\begin{proposition}\label{thm:frame_TR_in_NKCP3}
Let $M$ be a totally real surface in the nearly Kähler $\CP^3$ and consider the functions of \Cref{eq:J0def_general}, i.e. $b_1, b_2, b_4, b_5, c_3, c_4$ and $c_5$. By restricting to an open dense subset, there exists a frame such that either
\begin{enumerate}
  \item[0)]  $M$ is of Type 0, i.e. there are local functions $\lambda$, $\alpha$ and $\theta$ such that
   \begin{align*}
        b_1 &= \tfrac{\lambda}{1+\lambda^2}(1-\sin \alpha), & b_2 &=\tfrac{\lambda^2 +\sin \alpha}{1+\lambda^2}, & b_4 &=\tfrac{\cos \alpha \sin \theta}{\sqrt{1+\lambda^2}},  & b_5&=\tfrac{\cos \alpha \cos \theta}{\sqrt{1+\lambda^2}}, \\
        c_3&= \tfrac{1+\lambda^2 \sin \alpha}{1+\lambda^2},
        & c_4&=\tfrac{\lambda \cos \alpha \cos \theta}{\sqrt{1+\lambda^2}}, & c_5&=-\tfrac{\lambda \cos \alpha \sin \theta}{\sqrt{1+\lambda^2}}. & &
  \end{align*}

    Moreover, at any point, $\alpha$ can be chosen to lie in $]-\tfrac \pi{2},\tfrac{\pi}{2}[$.
    \item[1)] $M$ is of Type 1, i.e. we have that
    \begin{equation*}
    b_1=b_2=c_3=\tfrac 12 \qquad
    b_4=-c_5=\tfrac 1{\sqrt{2}} \qquad
    b_5=c_4=0.
    \end{equation*}
    \item[2)] $M$ is of Type 2, i.e. there exists a local function $\alpha$ such that 
    \begin{equation*}
        b_1 =\cos \alpha, \qquad
        b_2 =\sin \alpha, \qquad
        c_3= -\sin \alpha, \qquad
        b_4 =b_5=c_4=c_5=0.
    \end{equation*}
    \end{enumerate}
    Moreover, there is no overlap between the three types, and $\FSm(U,V) = 0$ in any case.
\end{proposition}
\begin{proof} 
We first assume that $b_4^2+b_5^2 \ne 0$. 
It first follows from $g(\Jo U,Jo V)=g(U,V)=0$ that $b_4 c_4 +b_5 c_5=0$. Hence there exists 
a local function $\lambda$ such that $c_4 =\lambda b_5$ and $c_5 =-\lambda b_4$.

Next we use the fact that $\Jo^2 U=-U$ and $\Jo^2 V=-V$ are orthogonal to $G(U,V)$ and $JG(U,V)$.
This gives us the following system of equations.
\begin{align}&b_1^2+b_2-b_2 c_3-b_1 \lambda=0\label{eqGuv}\\
&-b_1 +b_1^2 \lambda +c_3 \lambda -b_2 c_3 \lambda =0.\label{eqjGuv} 
\end{align}
Combining these equations it immediately follows that $b_1(1-\lambda^2)=- (b_2-c_3) \lambda$.

Now we consider several subcases. 
\begin{enumerate}
\item Assume that $\lambda(1-\lambda^2) \ne 0$. It then follows that $b_1=\tfrac{(b_2-c_3)\lambda}{\lambda^2-1}$.
Substituting this expression into \Cref{eqGuv}, it follows that
either $c_3= 1-\lambda^2+b_2 \lambda^2$ or $c_3 =\tfrac{b_2}{\lambda^2}$. We consider these seperately.
\begin{enumerate}
    \item  $c_3= 1-\lambda^2+b_2 \lambda^2$. Given that 
    $\Jo^2 U=-U$, it follows that $b_2^2+b_4^2+b_5^2=1-(1-b_2)^2 \lambda^2$, so that we can write $b_4 = \tfrac{\mu \sin \theta}{\sqrt{1+\lambda^2}}$ and $b_5 = \tfrac{\mu \cos \theta}{\sqrt{1+\lambda^2}}$. Using this, the previous equation becomes $b_2^2 (1+\lambda^2)-2 b_2 \lambda^2 +\tfrac{(\lambda^4-1)+\mu^2}{\lambda^2+1}=0$.
    Computing its discriminant (seen as a polynomial in $b_2$), we see that we can write $\mu= \cos \alpha$, with $\alpha$ (at any point) taking values in $]-\tfrac \pi 2,\tfrac \pi 2[$ and solving for $b_2$ we find that $b_2= \tfrac{\lambda^2+\sin \alpha}{1+\lambda^2}$. Using now the previous expressions for $b_1,c_3,b_4,b_5,c_4,c_5$ we see that we obtain a totally real surface of Type 0.
    \item  $c_3 =\tfrac{b_2}{\lambda^2}$. We proceed in the same way as in the previous case. Given that 
    $\Jo^2 U=-U$, it follows that $b_2(b_2-(b_4^2+b_5^2)\lambda^2)^2 =0$ and $b_2(b_2+(b_4^2+b_5^2)(1+\lambda^2) =\lambda^2$.
    From this, it follows that $b_2=(b_4^2+b_5^2) \lambda^2$ and $(b_4^2+b_5^2)= (1+\lambda^2)$. So we can write  $b_4 = \tfrac{\sin \theta}{\sqrt{1+\lambda^2}}$ and $b_5 = \tfrac{\cos \theta}{\sqrt{1+\lambda^2}}$.
    Using now the previous expressions for $b_1,c_3,b_4,b_5,c_4,c_5$ we see that we obtain a totally real surface of Type 0 with $\alpha = 0$. 
\end{enumerate}
\item $\lambda^2=1$. In this case, if necessary by interchanging $U$ and $V$ we may assume that $\lambda=1$. From the previous equation, we see that $c_3 = b_2$. Note that as in this case $J\Jo$ restricted to the tangent space is a multiple of the identity, the frame along the totally real surface still admits a rotational freedom. Therefore by applying a rotation we may assume that $b_5=c_4=0$. The equation \eqref{eqGuv} reduces to
$b_1^2-b_1=b_2^2-b_2$,
which implies that either $b_2=b_1$ or $b_1=1-b_2$.
\begin{enumerate}
    \item $b_1=1-b_2$. The fact that $\Jo$ preserves the length of the vector field $U$ now implies that 
    $(1-b_2)^2+b_2^2 +b_4^2=1$,
which reduces to 
$2b_2^2 -2 b_2 +b_4^2=0$.
By looking at the discriminant of this equation we see that we can write $b_4 = \tfrac{1}{\sqrt{2}}\cos \alpha$ and $b_2 =\tfrac{1+\sin \alpha}{2}$. 
Hence we deduce that $M$ is a totally real surface of Type 0 with $\lambda=1$ and $\theta = \tfrac{\pi}{2}$.
\item $b_1=b_2$. Here we use both the fact that $\Jo U$ is a vector with length $1$ as well as the fact that $\Jo^2 U=-U$. This gives us the following system of equations:
\begin{align*}
    2b_1^2+b_4^2 &=1, &
    b_1(b_1-b_4^2) &=0, &
    2 b_1(b_1+b_4^2) &=1.
\end{align*}
Solving this system of equations we get that $b_2 =b_4^2=\tfrac 12$. If necessary by changing the sigh of both $U$ and $V$ we may then assume that $b_4 = \tfrac{1}{\sqrt{2}}$. We conclude now that $M$ is  a totally real surface of Type 1. 
\end{enumerate}
\item $\lambda=0$. This implies that $b_1=0$.  Using again the fact that $\Jo$ preserves the lengths of vectors and that $\Jo^2 U=-U$, we get the following equations:
\begin{align*}
    c_3^2&=1 &
    b_2^2+b_4^2+b_5^2&=1 &
    b_2(1-c_3)&=0 &
    b_2^2+(b_4^2+b_5^2)c_3&=1.
\end{align*}
  In order for this system to have solutions we must have that $c_3=1$ and there has to exist angles $\alpha$ and $\theta$ such that $b_2= \sin \alpha$, $b_4=\cos \alpha \sin \theta$ and $b_5= \cos \alpha \cos \theta$ which corresponds to a surface of Type $0$ with $\lambda=0$.   
\end{enumerate}
Finally, we consider the case that $b_4=b_5=c_4=c_5=0$. Using again the fact that $\Jo$ preserves the lengths of vectors and that $\Jo^2 U=-U$, we get the following equations:
\begin{align*}
    b_1^2+b_2^2 &=1 &
    b_1^2+c_3^2 &=1 &
    b_1(b_2+c_3) &=0.
\end{align*}
So we see that either $b_1=0$ and ,$b_2=c_3=1$ or $b_2=c_3=-1$ $b_2=-c_3= \sin \alpha$ and $b_1=\cos \alpha$.  The last possibility gives a totally real surface of Type 2. The first two cases can be excluded by the following argument. Note that in those cases $\Jo=\epsilon J$ on the tangent space. 
As the second fundamental form is symmetric, it therefore also follows that
$\Jo(\nabla_U V-\nabla_V U) =\epsilon J(\nabla_U V-\nabla_V U))$.
Therefore, we have
\begin{equation*}
    (\nabla_U \Jo) V- (\nabla_V \Jo) U 
    =\epsilon ( (\nabla_U J) V- (\nabla_V J) U) 
    =\epsilon (G(U,V)-G(V,U)) 
    =2 \epsilon G(U,V),
\end{equation*}
from which a contradiction follows immediately. 

We already remarked that $\FSm(U,V) = 0$, from the fact that we diagonalised $K$ on $M$. 
\end{proof}
\begin{remark}
    Note that in the previous theorem, totally real surfaces of Type 2 could be considered as a special case of totally real surfaces of Type 0 (with $\alpha = -\pi/2$). However, the surfaces of Type 2 most often lead to special cases in the fundamental equations, making it computationally beneficial to consider it a separate case. Geometrically, Type 2 surfaces are expected to be special cases because of the following. For every totally real surface $M$, the subspace of $T\CP^3$ containing $TM$ and invariant under $J$ and $\Jo$ is either four- or six-dimensional. The surfaces of Type 2 are exactly those for which the minimum of four is reached and $\Jo TM \subset TM \oplus J TM$. 
    
    Additionally, it is probable that there exist surfaces of Type 1, as the system is not overdetermined. However, every natural additional condition on the surface immediately leads to an overdetermined contradictory system. 
\end{remark}

The totally real surfaces of the nearly Kähler $\CP^3$ are linked to totally real and almost complex surfaces in Kähler $\CP^3$ as follows. 
\begin{proposition}\label{thm:Kahler_TR_AC_type}
    A totally real surface of nearly Kähler $\CP^3$ is totally real for $\Jo$ if and only if it is of Type 0 with $\lambda=0$ or of Type 2 with $\cos \alpha=0$.
    Moreover, it is almost complex for Kähler $\CP^3$ if and only if it is of Type 2 with $\sin \alpha=0$.
\end{proposition}
\begin{proof}
    A nearly Kähler totally real surface is totally real for $\Jo$ if and only if $\FSm(\Jo U, V)=0$. We see immediately that Type 1 is then impossible. For Type 2, we see immediately that this condition is equivalent to $\cos \alpha$ vanishing. If the surface is of Type 0, we have that $\FSm(\Jo U,V)$ is proportional to $\lambda (\sin\alpha-1)$. Since $\alpha = \pm \pi/2$ is excluded from Type 0, we have to conclude that $\lambda = 0$ is the only option.

    If the surface is almost complex for Kähler $\CP^3$, then, and only then, $\Jo U_0 = \pm V_0$. Here, $U_0$ and $V_0$ are $U$ and $V$ rescaled to be of $\FSm$-unit norm, which are also orthogonal to each other. For Types 0 and 1, it is immediate that this condition is never satisfied. For Type 2, we find that $\sin \alpha$ has to vanish.
\end{proof}

% \begin{proposition}
%     If $M$ is a totally real surface of nearly Kähler $\CP^3$ that is immersed in a Lagrangian submanifold with $\theta^{\mathcal{L}}=0$, then one of the following must hold.
%     \begin{itemize}
%         \item $M$ is of Type 0 with $\lambda = 0$ and $\cos \theta = 0$.
%         \item $M$ is of Type 2 with $\cos\alpha = 0$.
%     \end{itemize}
% \end{proposition}
% \begin{proof}
    
% \end{proof}
% \todo[inline]{the Type 2 are (potentially) exactly the Codazzi like examples}
% \todo[inline]{the Type 0 (potentially) includes minimal examples }

%%%%%%%%%%%%%%%%%
\section{Intermezzo: Legendre curves in spheres}\label{sec:legendre_curves}
%%%%%%%%%%%%%%%%%
In this section, we recall the notion of a Legendre curve (or $C$-totally real) in $\sphere{3} \subset \complexs^2$ as to introduce numerous examples of totally real surfaces later. Moreover, we also extend the notion of a Legendre curve to $\sphere{7} \subset \quaternions^2$.

\subsection{Legendre curves in \texorpdfstring{$\sphere{3} \subset \complexs^2$}{S³ ⊂ ℂ²}}\label{sec:legendre_S3}

Consider the unit three sphere embedded in $\complexs^2$. Recall that a curve $\gamma: I \to \sphere{3} \subset \complexs^2$ is called Legendre, or C-totally real, if and only if $\langle \gamma', i \gamma \rangle = 0$, see also \cite{chen1997}. If we suppose arc length parametrisation, then we also have $\langle \gamma', \gamma' \rangle = 1$, so that $\set{ \gamma, \gamma', i \gamma, i \gamma'}$ forms an orthonormal frame along $\gamma$. It follows that $\gamma$ satisfies the differential equation
\begin{equation}\label{eq:legendre_S3}
    \gamma'' = - \gamma + \kappa i \gamma',
\end{equation}
for some function $\kappa$, which may be thought of as a curvature. 

Conversely, suppose we have an arc length parametrised curve $\gamma: I \to \sphere{3}$ that satisfies \Cref{eq:legendre_S3}. Suppose as initial conditions that $\set{ \gamma, \gamma', i \gamma, i \gamma'}$ forms an orthonormal frame at one point of $\gamma$. Then, we can compute that
\begin{equation}\label{eq:system_converse_legendre_S3}
    \begin{cases}
        \langle \gamma,\gamma \rangle' &= 2 \langle \gamma,\gamma' \rangle \\
        \langle \gamma,\gamma' \rangle' &=  \langle \gamma',\gamma' \rangle - \langle \gamma,\gamma \rangle + \kappa_1 \langle \gamma, i \gamma' \rangle \\
        \langle \gamma,i \gamma' \rangle' &= \langle \gamma',i \gamma' \rangle - \langle \gamma, i \gamma \rangle  - \kappa_1 \langle \gamma,  \gamma' \rangle \\
        \langle \gamma',\gamma' \rangle' &= - 2 \langle \gamma',\gamma \rangle  + 2 \kappa_1 \langle \gamma', i \gamma' \rangle.
    \end{cases}
\end{equation} This is a first order homogeneous system of differential equations in the functions $\langle \gamma,\gamma \rangle$, $\langle \gamma,\gamma' \rangle$, $\langle \gamma,i \gamma' \rangle$ and $\langle \gamma',\gamma' \rangle$ and so has a unique solution, given initial conditions. As $\langle \gamma,\gamma \rangle=1$, $\langle \gamma,\gamma' \rangle=0$, $\langle \gamma,i \gamma' \rangle=0$ and $\langle \gamma',\gamma' \rangle=1$ is a trivial solution that agrees with the initial conditions, we find that this is the unique solution to \Cref{eq:system_converse_legendre_S3}. Hence, $\langle \gamma', i \gamma \rangle = 0$ for all time, and $\gamma$ is a Legendre curve in $\sphere{3}$.

\subsection{Legendre curves in \texorpdfstring{$\sphere{7} \subset \quaternions^2$}{S⁷ ⊂ ℍ²}} \label{sec:legendre_S7}

The well-known ideas of \Cref{sec:legendre_S3} can be generalised naturally to $\sphere{7} \subset \quaternions^2$. This hasn't been done in the literature, as far as the authors are aware. 

We define a Legendre curve in $\sphere{7}$ as a curve $\gamma: I \to \sphere{7} \subset \quaternions^2$ with $\langle \gamma', i \gamma \rangle = \langle \gamma', j \gamma \rangle = \langle \gamma', k \gamma \rangle = 0$. Again, supposing arc length parametrisation, we have that \[\set{ \gamma, i \gamma, j \gamma, k \gamma, \gamma', i \gamma', j \gamma', k \gamma'}\] forms an orthonormal frame along $\gamma$. It follows that $\gamma$ satisfies the differential equation
\begin{equation}\label{eq:legendre_S7}
    \gamma'' = - \gamma + \kappa_1 i \gamma'+ \kappa_2 j \gamma'+ \kappa_3 k \gamma',
\end{equation}
for some curvature functions $\kappa_1$, $\kappa_2$ and $\kappa_3$.

With the same argument as in \Cref{sec:legendre_S3}, we find that the converse is also true: given a curve $\gamma: I \to \sphere{7} \subset \quaternions^2$ that is arc length parametrised, satisfies \Cref{eq:legendre_S7} and with initial conditions that $\set{ \gamma, i \gamma, j \gamma, k \gamma, \gamma', i \gamma', j \gamma', k \gamma'}$ forms an orthonormal frame at one point of $\gamma$, then $\gamma$ has to be a Legendre curve of $\sphere{7}$.

%%%%%%%%%%%%%%%%%
\section{Examples of totally real surfaces}\label{sec:examples}
%%%%%%%%%%%%%%%%%

In this section, we present numerous examples of totally real surfaces of the nearly Kähler $\CP^3$. We will do so using lifts to $\sphere{7} \subset \reals^8 \cong \complexs^4$ under the Hopf fibration $\hpi$. Recall that totally real submanifolds of the Kähler $\CP^3$ can be horizontally lifted to $\sphere{7}$, see \cite{reckziegel1985}. If the submanifolds are not totally real for the Kähler $\CP^3$, the following lifts are not horizontal.

We remark that almost all examples are full submanifolds (they do not lie in any totally geodesic submanifolds of the nearly Kähler $\CP^3$). Indeed, in recent work \cite{lorenzo-naveiro2024}, the totally geodesic submanifolds of the nearly Kähler $\CP^3$ were classified. Apart from some almost complex spheres, only the Lagrangian $\reals P^3$ can occur as a totally geodesic submanifold of the nearly Kähler $\CP^3$. As introduced in \cite{liefsoens2024}, there is an angle function $\theta^L$ associated with all Lagrangians in the nearly Kähler $\CP^3$, and the totally geodesic $\reals P^3$ is situated at $\theta^L = 0$. If a totally real surface is immersed into a Lagrangian, the tangent space of the Lagrangian is necessarily spanned by $\set{U,V, JG(U,V)}$, so we can compute the angle function $\theta^L$. By considering when this angle function is identically zero, we can show that many examples are full. Unless stated explicitly below, all examples are full. 

\subsection{Extrinsically homogeneous examples}

In this section, we consider extrinsically homogeneous totally real surfaces.
Recall that these are submanifolds which are the orbits under a subgroup of the identity component of the isometry group. In the case of the nearly Kähler $\CP^3$, they are the orbits under subgroups of $\Sp(2)$. The following four families are flat and orbits under $\U(1) \oplus \U(1)$ subgroups of $\Sp(2)$. We take the orbit of the point $p_0 = \mqty(1 & 0 & 0 & 0)^T \in \complexs^4$. The last example is a $\SU(2)$ orbit, and is a sphere.

\begin{description}
  \item[Family 1] 
  For all parameters $\mu, \nu \in \reals$, consider the immersion 
    \begin{equation}
        \label{eq:EH_tori_fam1}
        \begin{aligned}
            &[0, 2 \pi] \times [0, 2 \pi] \to \CP^3: \\
            &(t,s) \mapsto \hpi\l( 
            \mqty(
            \cos{s} \cos{t}-\frac{\mu \nu  \sin{s} \sin{t}}{\sqrt{\nu ^2+1} \sqrt{1 + \mu^2+\nu ^2}} \\
            -\frac{\sin{s} \cos{t}}{\sqrt{1 + \mu^2+\nu ^2}} \\
            -\frac{\sqrt{\nu^2+1} \sin{s} \sin{t}}{\sqrt{1 + \mu^2+\nu ^2}} \\
            -\frac{\cos{s} \sin{t}}{\sqrt{\nu ^2+1}}
            ) 
            + i \mqty( 
            -\frac{\mu \sin{s} \cos{t}}{\sqrt{1 + \mu^2+\nu ^2}}-\frac{\nu  \cos{s} \sin{t}}{\sqrt{\nu ^2+1}} \\
            0 \\
            -\frac{\nu  \sin{s} \cos{t}}{\sqrt{1 + \mu^2+\nu ^2}} \\
            \frac{\mu \sin{s} \sin{t}}{\sqrt{\nu ^2+1} \sqrt{1 + \mu^2+\nu ^2}}
            ) \r). 
           \end{aligned}
        \end{equation}
  \item[Family 2] 
    Consider the parameter $\nu \in \reals$ and the following immersion: 
    % \begin{equation}
    %     \label{eq:EH_tori_fam2a}
    %         [0, 2 \pi] \times [0, 2 \pi] \to \CP^3: 
    %         (t,s) \mapsto \hpi
    %         \mqty(
    %         \frac{\cos (s)+\nu ^2 e^{-i t}}{\nu ^2+1} \\
    %          -\frac{\sin{s}}{\nu ^2+1} \\
    %          -\frac{i \nu  \sin{s}}{\nu ^2+1} \\
    %          \frac{i \nu  \left(\cos (s)-e^{-i t}\right)}{\nu ^2+1}
    %         ). 
    %     \end{equation}
        \begin{equation}
        \label{eq:EH_tori_fam2a}
            [0, 2 \pi] \times [0, 2 \pi] \to \CP^3: 
            (t,s) \mapsto \hpi
            \mqty(
            \frac{\cos (s)+\nu ^2 e^{-i t}}{\nu ^2+1} &
             -\frac{\sin{s}}{\nu ^2+1} &
             -\frac{i \nu  \sin{s}}{\nu ^2+1} &
             \frac{i \nu  \left(\cos (s)-e^{-i t}\right)}{\nu ^2+1}
            )^T. 
        \end{equation}
    \item[Family 3] 
    Consider the parameters $\mu, \nu \in \reals$, suppose $\mu \neq 0$ and let $f = \sqrt{\mu^2 + (1+\nu^2)^2}$. We have the following family of immersions: 
    \begin{equation}
        \label{eq:EH_tori_fam2b}
            [0, 2 \pi] \times [0, 2 \pi] \to \CP^3: 
            (t,s) \mapsto \hpi
            \mqty(
            \frac{\nu ^2 e^{-i \left(\frac{\mu s}{f}+t\right)}+\frac{\mu e^{-i s}}{f}}{\nu ^2+1} \\
         -\frac{\sin{s}}{f} \\
         -\frac{i \nu  \sin{s}}{f} \\
         -\frac{i \nu  \left(e^{-i \left(\frac{\mu s}{f}+t\right)}-\frac{\mu e^{-i s}}{f}\right)}{\nu ^2+1}
            ). 
        \end{equation}
    \item[Family 4] 
  Consider the parameters $\nu, \rho, \sigma \in \reals$, suppose $\sigma \neq 0$ and that $\mu$ is given by
  $8 \mu  \sigma +8 \nu ^2 \rho +2 \sqrt{2} \nu  \left(8 \rho ^2+8 \sigma ^2-1\right)-8 \rho = 0$.
  
  We have the following family of immersions: 
    \begin{equation}
        \label{eq:EH_tori_fam3}
            [0, 2 \pi] \times [0, 2 \pi] \to \CP^3:
            (t,s) \mapsto \hpi
            \mqty(
            e^{t X} e^{s Y} p_0
            ),
        \end{equation}
        where $X$ and $Y$ are the following matrices
        \begin{equation*}
            X = 
            \mqty(
            -\frac{i \nu }{\sqrt{2}} & 0 & 0 & \frac{1}{\sqrt{2}} \\
             0 & \frac{i \nu }{\sqrt{2}} & -\frac{1}{\sqrt{2}} & 0 \\
             0 & \frac{1}{\sqrt{2}} & 2 i \rho & 2 \sigma \\
             -\frac{1}{\sqrt{2}} & 0 & -2 \sigma & -2 i \rho         
            ), \qquad Y = \mqty(
            -i \mu & 1 & -i \nu  & 0 \\
             -1 & i \mu & 0 & i \nu  \\
             -i \nu  & 0 & i \left(\mu+2 \sqrt{2} \sigma \nu \right) & 1-\nu  \left(2 \sqrt{2} \rho+\nu \right) \\
             0 & i \nu  & 2 \sqrt{2} \rho \nu +\nu ^2-1 & -i \left(\mu+2 \sqrt{2} \sigma \nu \right)
            ).
        \end{equation*}
    \item[Sphere] 
    The following immersion gives a sphere in $\CP^3$:
    \begin{equation}\label{eq:ex_sphere_in_Lagrangian}
        [0, \pi] \times [0, 2 \pi ] \to \sphere{7} \subset \complexs^4 :   (u,v) \mapsto  \hpi \mqty( \cos u &
     \frac{\sqrt{2}}{\sqrt{3}} i e^{i v} \sin u &
    0 &
     \frac{1}{\sqrt{3}} e^{i v} \sin u
    )^T.
\end{equation}
\end{description}
%$h_2 = \rho$
%$h_3 = \sigma$

We now discuss each family in the following sections.

\subsubsection{Family 1}
The parameter $\mu$ can be recognised as $g(H, JV)$ for $\nu=0$ and with $H$ the mean curvature vector. Except for $\nu = 0$, all the resulting surfaces are of Type 0 with $\lambda=0$ and $\theta = 0$. For $\nu = 0$, the surfaces are of Type 2 with $\cos \alpha = 0$. When $\nu = \mu=0$, the example is not full.

We also compute the mean curvature vector with respect to the nearly Kähler metric ($H$), and the Kähler metric ($H_\circ$). We lift $H$ and $H_\circ$ at $\hpi(p_0)$ horizontally to have 
\begin{align*}
    H &= \mqty( 0 & \frac{h_1 \nu }{2 \nu ^2+1} & \frac{i h_1}{2 \nu^2+1} & \frac{i \nu ^3}{2 \nu ^2+1} )^T, &
    H_\circ &=  \mqty( 0 & \frac{h_1 \nu }{\nu ^2+1} & \frac{i h_1}{\nu^2+1} & i \nu )^T.
\end{align*}
Notice that the resulting surface is minimal (for either Kähler or nearly Kähler $\CP^3$) if and only if $\nu=h_1 = 0$.
Then, this surface is the Clifford torus in $\reals P^3$, and an embedding is given by 
\begin{equation}\label{eq:EH_tori_fam1_minimal_parallel}
    [0, 2 \pi] \times [0, 2 \pi] \to \sphere{7}: (t,s) \mapsto 
    \hpi \mqty(
    \cos{s} \cos{t} &
    - \sin{s} \sin{t} &
    - \sin{s} \cos{t} &
    - \cos{s} \sin{t}
    )^T.
\end{equation}
Notice that this is the only parallel example of the family of \Cref{eq:EH_tori_fam1}.

\subsubsection{Family 2}
All of the immersions within Family 2 are of Type 0 with $\lambda = \theta = 0$. 

We can again compute the mean curvature vectors to find 
\begin{align*}
    H &= \mqty( 0 &
 0 &
 0 &
 \frac{i \l(4 \nu ^4-\nu ^2-1\r)}{4 \nu \l(2 \nu ^2+1 \r)} )^T, &
    H_\circ &=  \mqty( 0 &
 0 &
 0 &
 i\l ( \nu -\frac{1}{2 \nu }\r) )^T, 
\end{align*}
so that the only minimal examples occur at $\nu = \pm \frac{1}{2 \sqrt{2}} \sqrt{1 + \sqrt{17}}$. Moreover, there are no parallel examples within this family. 

Notice that for $2\nu^2 =1$ (and only then), there is a totally real surface that is minimal in the Kähler $\CP^3$. Write the corresponding immersion of Family 2 by $\hpi \circ f$ where $f$ is given by the argument of $\hpi$ in \Cref{eq:EH_tori_fam2a}. Without loss of generality, take $\nu = 1/\sqrt{2}$. The parameter $\alpha$ of \Cref{thm:frame_TR_in_NKCP3} is equal to zero for this example. 
By adding a factor $\exp(i t/3)$ to $f$, we have a horizontal lift of the unique totally real surface of Family 2 that is minimal for the Fubini-Study metric. Let $v_1, v_2, v_3$ be the following orthogonal vectors in $\complexs^4$:
\begin{align*}
    v_1&= \frac{1}{3} \mqty(  1 & i & -\frac{1}{\sqrt{2}} & \frac{i}{\sqrt{2}} )^T & v_2&= \frac{1}{3} \mqty(  1 & -i & \frac{1}{\sqrt{2}} & \frac{i}{\sqrt{2}} )^T & v_3&= \frac{1}{3} \mqty(  1 & 0 & 0 & - i\sqrt{2} )^T.
\end{align*}
Then, after reparametrising, the horizontal lift $e^{i t/3} f$ can be written as
\begin{equation}\label{eq:unique_TR_flat_minimal_KCP2}
    e^{i u} v_1 + e^{i v} v_2 + e^{-i(u+v)} v_3, \qquad u = \frac{t}{3}+s, v = \frac{t}{3}-s.
\end{equation}
This can be recognised as the unique totally real and flat minimal torus in the standard $\CP^2$, and with $\CP^2$ totally geodesically embedded in $\CP^3$, see \Cite{ludden_totally_1975}. %\Cite[example 3.1 P350]{kon_structures_1984}. 

\subsubsection{Family 3}
All of the immersions within Family 3 are of Type 0 with $\lambda = \theta = 0$. 

The mean curvature vectors are given by
\begin{align*}
    H &= \mqty( 0 &
 \frac{\mu \nu }{2 \nu ^2+1} &
 \frac{i \mu}{2 \nu ^2+1} &
 \frac{i \l(4 \nu ^4-\nu ^2-1\r)}{4 \nu \l(2 \nu ^2+1 \r)} )^T, &
    H_\circ &=  \mqty( 0 &
 \frac{\mu \nu }{\nu ^2+1} &
 \frac{i \mu}{\nu ^2+1} &
 i\l ( \nu -\frac{1}{2 \nu }\r) )^T, 
\end{align*}
so that there are no minimal examples in this family, as $\mu \neq 0$. There are no parallel examples within this family. 

\subsubsection{Family 4}

If $\nu = 0$, then, and only then, the resulting surface is of Type 2 with $\cos \alpha = 0$, and otherwise it is again of Type 0 with $\lambda = \theta = 0$. 

The (horizontal lifts of the) mean curvature vectors at $\hpi(p_0)$ are 
\begin{align*}
    H &= \mqty( 0 &
 \frac{2 h_1 \nu -\sqrt{2} h_3}{4 \nu ^2+2} &
 \frac{i h_1}{2 \nu ^2+1} &
 -i\frac{2 \sqrt{2} h_2-4 \nu ^3+\nu}{8 \nu ^2+4} )^T, &
    H_\circ &=  \mqty( 0 &
 \frac{h_1 \nu -\sqrt{2} h_3}{\nu ^2+1} &
 \frac{i h_1}{\nu ^2+1} &
 i \frac{-2 \sqrt{2} h_2+2 \nu ^3+\nu}{2 \l(\nu^2+1\r)} )^T,
\end{align*}
so that this family does not include any minimal examples (not for Kähler, nor nearly Kähler $\CP^3$). 

However, this family does include parallel examples (surfaces with $\nabla h=0$). They occur exactly when $\sigma \neq 0$ and $\nu = \rho = \mu =0$. All these examples are not full. Let $\kappa = -2 \sqrt{2} \sigma \neq 0$ and $f(\kappa) = 1 + \frac{1}{2} \kappa  \l (\sqrt{\kappa ^2+4}+\kappa \r)$. Notice that $f(\kappa)>0$ and $f(\kappa) \neq 1$, as $\kappa \neq 0$.
Then, explicitly, the parallel examples of \Cref{eq:EH_tori_fam3} are given by 
\begin{align}\label{eq:EH_tori_fam3_parallel}
    [0, 2 \pi] \times [0, 2 \pi] \to \reals P^3 \subset \CP^3: (t,s) \mapsto
    & \hpi( e^{-j s} \gamma_\kappa(t)),
\end{align}
where
\begin{equation*}
    \gamma_\kappa(t) = 
    \frac{1}{1+f(\kappa)}
    \mqty(
    \cos (t f(\kappa ))+f(\kappa ) \cos (t) \\
     \frac{\kappa  \sqrt{f(\kappa )} (\sin (t f(\kappa ))-f(\kappa ) \sin (t))}{f(\kappa )-1} \\
     \frac{\kappa  f(\kappa ) (\cos (t)-\cos (t f(\kappa )))}{f(\kappa )-1} \\
     \frac{\left(1 -f(\kappa )+\kappa ^2\right) f(\kappa )^2 \sin (t)-\left(\kappa ^2 f(\kappa )+f(\kappa )-1\right) \sin (t f(\kappa ))}{(f(\kappa )-1) \sqrt{f(\kappa )}}
    ) \in \sphere{3}(1) \subset \sphere{7}(1),
\end{equation*}
and $e^{j s} = \cos{s}+ j \sin{s}$.

Notice how $s \mapsto e^{-j s}\mqty(1,0,0,0)^T$ is a Legendre curve in $\sphere{3}$, if we consider $\complexs \cong \reals \oplus \reals j$. Moreover, the curve $t \mapsto \gamma_\kappa(t/ \sqrt{f(\kappa)}) = \tilde{\gamma}_\kappa(t)$ is arc length parametrised and satisfies
\[ \tilde{\gamma}_\kappa''(t) = - \tilde{\gamma}_\kappa(t) + \kappa \, j \tilde{\gamma}_\kappa'(t),  \]
so that this (and its reparametrisation $\gamma_\kappa$) is also a Legendre curve (both in $\sphere{7}$ and $\sphere{3}$ after identification $\complexs \cong \reals \oplus \reals j$) with constant curvature $\kappa$. 

\subsubsection{Sphere}
Recall the Lagrangian $\sphere{2}\times \sphere{1}$ in the nearly Kähler $\CP^3$ that is CR for Kähler $\CP^3$. From the explicit immersion of \cite{liefsoens2024} for $t=0$, we find the immersion stated in \Cref{eq:ex_sphere_in_Lagrangian}. It is non-minimal and of Type 2 with $\sin \alpha = 0$ and thus almost complex for the Kähler $\CP^3$ (see \Cref{thm:Kahler_TR_AC_type}). Moreover, it is totally geodesic for Kähler $\CP^3$ and totally umbilical in nearly Kähler $\CP^3$.

\subsection{An example that is totally geodesic for Kähler \texorpdfstring{$\CP^3$}{ℂP³}}

Another example that is totally real in nearly Kähler $\CP^3$ and totally geodesic in Kähler $\CP^3$, apart from \Cref{eq:ex_sphere_in_Lagrangian}, is the following for all $c\in \reals$ fixed:
\begin{equation}\label{eq:ex_RP2_in_RP3}
     \reals P^2 \to \reals P^3 \subset \CP^3 : [(y_1, y_2, y_3)] \mapsto 
    \hpi \mqty( 
    y_1 & 
    y_2 & 
    \cos{c}\; y_3 & 
    \sin{c}\; y_3
    )^T,
\end{equation}
where $\sphere{2} \subset \reals^3 \to \reals P^2: (y_1, y_2, y_3) \mapsto [(y_1, y_2, y_3)]$ denotes the standard projection. 
This example is simply $\reals P^2 \subset \reals P^3$ and is of Type 0 with $\lambda = 0$. This example is not full.

\subsection{A large class of Codazzi-like examples}\label{sec:examples_CL}

In this section, we consider Codazzi-like (or Codazzi for short) totally real surfaces. Codazzi surfaces are those for which $\nabla h$ is totally symmetric, and is a generalisation of parallel ($\nabla h = 0$) and totally geodesic submanifolds ($h=0$).

Suppose $I$ is some interval containing zero. Let $I \to \sphere{3} \subset \complexs^2 : v \mapsto (f_1(v) + f_2(v) i , f_3(v) + f_4(v) i )$ be an arc length parametrised Legendre curve with curvature function $\kappa(v)$, and denote by $K(v)$ the integral $K(v) = \int_0^v \kappa(t) \,\dd t$. Suppose that for $v=0$, we have $f_1(0)-1=f_2(0)= f_3(0)= f_4(0)=0$, as well as $f_1'(0)=0=f_2'(0)= f_3'(0)-1= f_4'(0)=0$ .
Identifying $\complexs^2$ and $\quaternions$, we will write $f: I \to \sphere{3} \subset \quaternions: v \mapsto f_1(v) + f_2(v) i + f_3(v)j + f_4(v)k$. Let $\gamma : I \to \sphere{7} \subset \quaternions^2$ also be an arc length parametrised Legendre curve. 
Given these choices, define the immersion
\begin{equation}\label{eq:all_CL_TR}
    F : I \times I \to \sphere{7} : (u,v) \mapsto f(v) \gamma(u). 
\end{equation} We claim that with the appropriate identification of $\complexs^4$ with $\quaternions^2$, this gives rise to a totally real Codazzi surface in nearly Kähler $\CP^3$.

Write $F_u$ for $\pdv{F}{u}$, and similarly for $v$. Define the operator $j$ by
\begin{align*}
    j F &= e^{i K(v)} F_v, &
    j F_u &= e^{i K(v)} F_{uv},  &
    j F_v &= - e^{i K(v)} F, &
    j F_{uv} &=  - i e^{i K(v)} F_u.
\end{align*}
It takes a simple computation to show that $j$ is parallel, and so that it indeed comes from a quaternionic structure on $\complexs^4$. Hence, we can identify $\complexs^4$ with $\quaternions^2$ such that $j$ acts as above.

Notice that $\langle F_u, i F \rangle = \langle \gamma'(u), \overline{f}if \gamma \rangle = 0$, since $\overline{f}if \in \Im \quaternions$ and $\gamma$ is a Legendre curve. With a similar argument, $\langle F_v, i F \rangle=0$, so that $F$ is horizontal with respect to the Hopf fibration $\hpi$. By construction of $j$, the vector field $F_v$ lies in $\tilde{\mathcal{D}}_1^2$, so that $J \dd \hpi(F_v) = \dd \hpi(- i F_v) $. Moreover, $F_u$ lies in the span of $F_{uv}$ and $i F_{uv}$. Since $\langle F, F \rangle = 1$ and $\langle F, i F \rangle = 0$, it follows that $F_u$ is orthogonal to $F_v$ and $i F_v$. It then also follows that $J \dd \hpi(F_u) = \dd \hpi(i F_u) $. With this information, it is easy to show that $\hpi \circ F$ gives a totally real surface in nearly Kähler $\CP^3$, as well as in Kähler $\CP^3$. 

We finish this discussion by showing that $\hpi \circ F$ is Codazzi in nearly Kähler $\CP^3$. Notice that $h(F_u,F_u) = F_{uu}$, $h(F_u,F_v) = F_{uv}$ and $h(F_v,F_v) = F_{vv}$. Since coordinate derivatives commute, we then immediately find that $\hpi \circ F$ is Codazzi. 

Finally, we note that all of these examples are of Type 2 with $\cos \alpha= 0$. Moreover, they are all full, except for the case where the image of $\gamma$ lies in $\sphere{3}$ and $\kappa$ is constant. When $\Im \gamma \subset \sphere{3}$, has constant curvature, and with $\kappa = 0$, then these are congruent to the parallel examples of \Cref{eq:EH_tori_fam3_parallel}.

%%%%%%%%%%%%%%%%%
\section{Nearly Kähler totally real, Kähler almost complex surfaces}
%%%%%%%%%%%%%%%%%

Suppose $M$ is not only totally real for nearly Kähler $\CP^3$, but also almost complex for Kähler $\CP^3$. From \Cref{thm:Kahler_TR_AC_type}, we know that $M$ has to be of Type 2 with $\sin\alpha = 0$. The next result shows there is a unique such $M$. 

\begin{theorem}\label{thm:classification_AC_KCP3}
    A totally real surface of nearly Kähler $\CP^3$ that is also almost complex for Kähler $\CP^3$ has to be congruent to example \eqref{eq:ex_sphere_in_Lagrangian}. In particular, it is the $\CP^1 \cong \sphere{2}$ inside the Lagrangian $\sphere{2} \times \sphere{1}$ with Lagrangian angle $\theta^L = \pi/4$. 
\end{theorem}
\begin{proof}
    According to \Cref{thm:Kahler_TR_AC_type}, this surface can only occur in Type 2 with $\sin \alpha=0$. By interchanging $U$ and $V$, we can assume $\alpha = 0$. From \Cref{eq:derivatives_P_J0} applied to $U,U$ and $V,V$, we find $k_1=k_2=m_1=m_2=0$; $k=m_4=-k_4$; and $m-1/2=-1-m_3 = k_3$. We then look at the Ricci equation for $U,V,G(U,V)$ and $JG(U,V)$ to find $(1+2 k_3)^2 + 4 k_4^2 = 0$, from which it follows $k_3 = -1/2$ and $k_4=0$. Notice that then $M$ is totally umbilical. The remaining Gauss equation tells us this has to be a sphere of Gaussian curvature $3$. Moreover, considering the second fundamental form in Kähler $\CP^3$, as well its curvature there, we find that this is a totally geodesic $\CP^1$. 

    We now show that from the above data, there is only one option, up to congruence. 
    Let $p_0$ be any point in $\CP^3$ and $q_0$ be a point of $\sphere{7}$ such that $\hpi(q_0) = p_0$. Let $U_0 = \frac{2}{\sqrt{3}} U$, the tangent space at $p_0$ is spanned by the $\FSm$-orthonormal frame $\set{ U_0, \Jo U_0 }$. Let $\tilde{U}_0$ be the horizontal lift of $U_0$ through $q_0$. Take any geodesic through $p_0$ with tangent vector $\cos v \, {(U_0)}_{p_0} + \sin v \, {(\Jo U_0)}_{p_0}$ and lift this to $\sphere{7}$, so that we obtain the surface $(u,v) \mapsto f(u,v) = \cos u \, q_0 + \sin(u) \,( \cos v \, (\tilde{U}_0)_{q_0} +\sin v \, i (\tilde{U}_0)_{q_0} )$. We know that $\sqrt{2} G(U,V)$ lies in $\Dfour$ and is of unit $\FSm$-length. Hence, we can use the homogeneity of $\CP^3$ to suppose that $q_0 = (1,0)\in \quaternions^2$, and the isotropy to suppose that $\sqrt{2} i \tilde{G(U,V)}_{q_0}  = (0,1)\in \quaternions^2$. Using the isotropy a final time, we can suppose that $\tilde{U}_0$ has no component in $(k,0)$. Then, we can introduce angles $\theta_1$ and $\theta_2$ such that $(\tilde{U}_0)_{q_0} = ( \cos\theta_1 \, j, \sin \theta_1 ( \cos \theta_2 \, j + \sin \theta_2 \, k ) )$, as $U_0$ has to be orthogonal to $G(U,V)$ and $\Jo G(U,V)$. Demanding that the resulting $U_0 = \dd \hpi(\partial_u f)$ has the correct length for $g$, we find that $\sqrt{3} \cos \theta_1 = \sqrt{6} \sin\theta_1 = \sqrt{2}$. Next, we use that for tangent vectors $U,V$, it holds that $2 G(U,V) = (\FSnabla J)(U,V)-(\FSnabla J)(V,U)$ (see \cite{liefsoens2024}) to find that with $\sqrt{2} i \tilde{G(U,V)}_{q_0}  = (0,1)$, we need $\theta_2 = -\pi/2$. Hence, we get that $f$, up to congruence, has to be $f(u,v) = \cos u \, (1,0) + \sin u \,e^{i v} \frac{1}{\sqrt{3}} ( \sqrt{2} \, j, - k )$. We can always translate the coordinates, and we apply the transformation $v\mapsto v + \pi/2$ to obtain $f(u,v) = \cos u \, (1,0) + \sin u \,e^{i v} \frac{1}{\sqrt{3}} ( \sqrt{2} \, k, j )$. 
    This is exactly the immersion of \Cref{eq:ex_sphere_in_Lagrangian}.
\end{proof}

%%%%%%%%%%%%%%%%%
\section{Totally real surfaces that are totally geodesic for Kähler \texorpdfstring{$\CP^3$}{ℂP³}}
%%%%%%%%%%%%%%%%%
\begin{theorem}\label{thm:tg_in_KCP3}
    Let $M$ be a totally real surface of nearly Kähler $\CP^3$ that is also totally geodesic for Kähler $\CP^3$. If $M$ is non-minimal, it has to be congruent to example \eqref{eq:ex_sphere_in_Lagrangian}, i.e. a specific $\CP^1$. If $M$ is minimal (in nearly Kähler $\CP^3$), it corresponds to example \eqref{eq:ex_RP2_in_RP3}, which is a specific $\reals P^2$. 
\end{theorem}
\begin{proof}
    We consider each type separately.  For Type 1, we find the following false system of equations: $0 = g(h_0(U,V), JV) = m_1 - \frac{1}{4 \sqrt{2}}$ and $0 = g(h_0(V,V), JU) = m_1 + \frac{1}{2 \sqrt{2}}$. Hence, there is no such surface of Type 1. 

    For Type 2, the vanishing of the second fundamental form for Kähler $\CP^3$ implies that $k_1=k_2=k_4=m_1=m_2=m_4=k=0$, $m_3 = k_3= -\cos \alpha /2$ and $m=-\sin \alpha/2$. Using \Cref{eq:derivatives_P_J0}, we require 
    \[ 0 = g( (\nabla \Jo)(e_1,e_2)+(\nabla \Jo)(e_2,e_1) - G(\mathcal{P}_1 e_1, \mathcal{P}_2 e_2)-G(\mathcal{P}_1 e_2, \mathcal{P}_2 e_1), e_5)  = -2 \sin \alpha. \]
    From \Cref{thm:Kahler_TR_AC_type}, it follows that $M$ is almost complex for Kähler $\CP^3$, and \Cref{thm:classification_AC_KCP3} implies that $M$ is then congruent to \Cref{eq:ex_sphere_in_Lagrangian}.

    We finish the proof with assuming $M$ is of Type 0. From $0 = g(h_0(U,U), JV)-g(h_0(U,V), JU)$, we find that $\cos \theta = 0$, where we used that $\cos \alpha \neq 0$. Then, from $0 = g(h_0(U,V), JV)-g(h_0(V,V), JU)$, it follows that $\lambda = 0$. From \Cref{thm:Kahler_TR_AC_type}, it follows that $M$ is totally real for Kähler $\CP^3$. Since it is also totally geodesic, it has to be ($\SU(4)$-)congruent to an $\reals P^2$. Lifting this horizontally to $\sphere{5}$, which is totally geodesically embedded in $\sphere{7}$, it only remains to fix the quaternion structure on $\sphere{7}$ such that the surface has the correct properties with respect to nearly Kähler $\CP^3$. Then, we find that it has to be congruent to \Cref{eq:ex_RP2_in_RP3}.
\end{proof}

%%%%%%%%%%%%%%%%%
\section{Extrinsically homogeneous totally real surfaces}
%%%%%%%%%%%%%%%%%

In this section, we will prove the following theorem.
\begin{theorem}\label{thm:EH}
    Suppose $M$ is a totally real surface of nearly Kähler $\CP^3$ that is also extrinsically homogeneous. Then, one of the following must hold.
    \begin{itemize}
        \item $M$ is totally real for Kähler $\CP^3$, a flat torus and locally congruent to one of \Cref{eq:EH_tori_fam1,eq:EH_tori_fam2a,eq:EH_tori_fam2b,eq:EH_tori_fam3}, or
        \item $M$ is almost complex for Kähler $\CP^3$ and thus locally congruent to \Cref{eq:ex_sphere_in_Lagrangian}. 
    \end{itemize}
    Furthermore, if $M$ is minimal, it is locally congruent to the Clifford torus in $\reals P^3$, i.e. \Cref{eq:EH_tori_fam1} for $h_1 = \nu=0$, or to  \Cref{eq:EH_tori_fam2a} for $\nu=\pm \frac{1}{2 \sqrt{2}} \sqrt{1 + \sqrt{17}}$. If, instead, $M$ is minimal for Kähler $\CP^3$, it is either the Clifford torus, or the torus of \Cref{eq:unique_TR_flat_minimal_KCP2}. 
\end{theorem}

This result is the combination of \Cref{thm:EH_not_S2,thm:EH_only_S2_left} that we will prove below.

\begin{figure}[bht]
    \centering
    \begin{tikzcd}[column sep=1.5em]
                                           &                                      &                                          & \sp(2)                                               &                                    &                            \\
                                           & \su(2)\oplus\su(2) \arrow[rru, hook, orange, dotted] &                                          & \su(2) \oplus \u(1) \arrow[u, hook, blue,dashed] \arrow[ll, hook, orange, dotted] &                                    &\su(2) \arrow[llu, hook] \\
        \u(1)\oplus \u(1) \arrow[ru, hook] & \u(2) \arrow[u, hook]                & \su(2) \arrow[ru, hook, blue,dashed] \arrow[lu, hook] &                                                      & \u(1)\oplus \u(1) \arrow[lu, hook] & \u(1) \arrow[u, hook]      \\
        \u(1) \arrow[u, hook]              & \su(2) \arrow[u, hook]               & \u(1)\oplus\u(1) \arrow[lu, hook]        & \u(1) \arrow[lu, hook]                               & \u(1) \arrow[u, hook]              &                            \\
                                           & \u(1) \arrow[u, hook]                & \u(1) \arrow[u, hook]                    &                                                      &                                    &                           
        \end{tikzcd}
\caption{Up to $\Ad_g: \sp(2) \to \sp(2): X \mapsto g X g^{-1}$ for some $g \in \Sp(2)$, this diagram gives all Lie sub algebras of $\sp(2)$. Modified from \cite{kerr2013}, where we used $\so(5) = \sp(2)$. Note that the chain $\SU(2) \U(1) \to \SU(2)\times\SU(2) \to \Sp(2)$ (whose corresponding Lie algebra sequence is indicated with orange dotted arrows) gives rise to the twistor fibration $\ttau: \CP^3 \to \sphere{4}$. Moreover, the chain $\SU(2) \to \SU(2)\times\U(1) \to \Sp(2)$ (corresponding sequence is indicated with blue dashed arrows) gives rise to the Hopf fibration $\hpi: \sphere{7} \to \CP^3$.  }
    \label{fig:sub_algebras_sp2}
\end{figure}

\begin{proposition}\label{thm:EH_only_S2_left}
    Suppose $M$ is a totally real surface of nearly Kähler $\CP^3$ that is extrinsically homogeneous and such that $T M \cap \Dtwo$ and $T M \cap \Dfour$ are zero-dimensional. Then, $M$ is almost complex for Kähler $\CP^3$ and thus locally congruent to \Cref{eq:ex_sphere_in_Lagrangian}. 
\end{proposition}
\begin{proof}
    Suppose $K \subset \Sp(2)$ is a subgroup such that $M = K \cdot p$, and let $K_0 = H \cap K$, with $H = \SU(2)\times U(1)$ the isotropy. Let $\mathfrak{k}, \mathfrak{k}_0$ be the lie algebras of $K$ and $K_0$, respectively. By \Cref{fig:sub_algebras_sp2}, we know that all options for $\mathfrak{k}_0$ are given by the second column in \Cref{table:possible_lie_algebras_EH_surfaces}. Since we are looking at surfaces in $\CP^3$, we need $\dim(\mathfrak{k}) - \dim(\mathfrak{k}_0) = 2$. Hence, we get the possibilities for $\mathfrak{k}$ listed in the fourth column of \Cref{table:possible_lie_algebras_EH_surfaces}. 

    In the remainder of the proof, we will show that only the case $\mathfrak{k}_0 = \u(1)$ and $\mathfrak{k} = \su(2)$ leads to a totally real example, and that they give rise to the example of \Cref{eq:ex_sphere_in_Lagrangian}. To do so, we will explicitly compute in $\sp(2)$ that only the above case is possible, by using the assumption that there is no non-trivial overlap between $TM$ and any one of the distributions. 

    We thus assume that 
    $U_p = m_1 + \lambda m_3$ and $V_p = a m_1 + b m_2 + \sum_{i=3}^6 \mu_i m_i$, where we also used the isotropy at $p$. We also know that $\lambda \neq 0$, $a^2+b^2 >0$ and $\mu_1^2 + \mu_2^2 + \mu_3^2 + \mu_4^2 >0$. By subtracting $a U$ from $V$, we may assume $a=0$. Assuming $U$ and $V$ are orthogonal, we find $\mu_3 = 0$. By renormalising $V$, we find $V_p = m_2 + \sum_{i=4}^6 \mu_i m_i \qquad \lambda, \mu_4^2 + \mu_5^2 + \mu_6^2 \neq 0$.

    Now, we force $M$ to be totally real, i.e. $g(J U_p, V_p) = 0$, to obtain $1 - \lambda \mu_5 = 0$. Note that then $\mu_5$ can no longer be zero. Hence, after renaming, we get $V_p = m_2 + \frac{1}{\lambda} m_5 + \mu e_4 + \nu e_6$ and $\lambda \neq 0$.
    Notice that $M$ is almost complex for Kähler $\CP^3$ if and only if $\lambda^2=1, \mu=0, \nu =0$. We lift these vectors to $\sp(2)$ to find that we need to consider the matrices of the form 
    \begin{equation}\label{eq:proof_EH_XY_lifted}
        X = m_1 + \lambda m_3 + \sum_{i=1}^4 a_i h_i 
    \qquad 
    Y = m_2 + \frac{1}{\lambda} m_5 + \mu e_4 + \nu e_6 + \sum_{i=1}^4 b_i h_i \qquad \lambda \neq 0.
    \end{equation}

    From the argument above and \Cref{table:possible_lie_algebras_EH_surfaces}, there are five cases to consider:
    \begin{enumerate}
        \item $X$ and $Y$ are commutative;
        \item $\{X,Y, [X,Y]\}$ span a $\SU(2)$ and $[X,Y] \in \mathfrak{h}$;
        \item There exist commuting $Z,W \in \mathfrak{h}$ such that $\{X,Y, Z, W\}$ generate a $\su(2) \oplus \u(1)$;
        \item There exist commuting $Z,W \in \mathfrak{h}$ such that $\{X,Y, Z, W\}$ generate a $\u(2)$;
        \item There exist $Z_1, Z_2, Z_3, W \in \mathfrak{h}$ such that $\{Z_1,Z_2,Z_3, W\}$ generate $\su(2)\oplus\u(1)$ and $\{X,Y, Z_1,Z_2, Z_3, W\}$ generate $\su(2) \oplus \su(2)$.
    \end{enumerate}
    Notice that the Lie algebras $\su(2) \oplus \u(1)$ and $\u(2)$ are isomorphic, so we can safely ignore case 4. Moreover, we can also ignore case 3, as the non-existence of other examples is implied once we have shown that in case 2. We now show that cases 1 and 5 give no totally real examples, and that case 2 only gives the example of \Cref{eq:ex_sphere_in_Lagrangian}. Then, we have proved the statement.

    \textit{Case 1}. Given $X,Y$ of \Cref{eq:proof_EH_XY_lifted}, we can take the commutator in $\sp(2)$ and project this to both $\h$ and $\m$. Splitting as such makes the equations easier to deal with. It is straightforward to check that these equations are never zero, for any values of $a_i$, $b_i$, $\lambda, \mu$ or $\nu$. Hence, this case cannot occur.

    \textit{Case 2}. Let $Z= [X,Y]_\h$. We demand that $X,Y,Z$ form an $\su(2)$ algebra. It is easy to check that $\kappa([Z,X], X)=\kappa([Z,X], Z) = 0$ and $\kappa([Z,Y], Y)=\kappa([Z,Y], Z) = 0$. 
    Let $\tilde{X}$ be some multiple of $X$, and similarly for $Y$ and $Z$.
    By considering $([\tilde{Z},\tilde{X}]-\tilde{Y})_\m$ and $([\tilde{Y},\tilde{Z}]-\tilde{X})_\m$, we find by which factor to rescale $X$, $Y$ and $Z$ with to obtain $[X,Y]=Z$, $[Y,Z]=X$ and $[Z,X]=Y$, if $X$, $Y$ and $Z$ form an involutive distribution. We thus find that we need to rescale $X$ and $Y$ with a factor $12^{-1/2}$, and $Z$ with a factor $12^{-1}$. On the other hand, in general, these factors are given by $\alpha^{-1/2}$, $\beta^{-1/2}$ and $(\alpha \beta)^{-1/2}$, where 
    $\alpha = \frac{\kappa( [Z,X], Y)}{\kappa(Y,Y)}$ and $\beta = \frac{\kappa( [Y,Z], X)}{\kappa(X,X)}$. 
    We thus have the following equations to solve:
    \[ \frac{\kappa( [Z,X], Y)}{\kappa(Y,Y)} = \frac{\kappa( [Y,Z], X)}{\kappa(X,X)} = 12, \quad [X,Y]=Z, \quad [Y,Z]=X, \quad [Z,X]=Y, \]
    in the unknowns $a_i, b_i, \lambda, \mu, \nu$. 
    Solving these equations, we find that only the following two solutions exist: $(a_i) = \mqty(0&0&0&-1)$, $(b_i) = \mqty(0&0&1&0)$, $\mu=\nu=0$ and $\lambda = \pm 1$. The conditions on $\mu, \nu, \lambda$ are equivalent with $M$ being almost complex for Kähler $\CP^3$, as we already remarked before. From \Cref{thm:classification_AC_KCP3}, we are done in this case.
    
    \textit{Case 5}. Note that $\{Z_1, Z_2, Z_3, W\}$ generate the full isotropy. So, we need to check that there are no $X,Y$ that together with the full isotropy forms a $\su(2) \oplus \su(2)$, and is totally real. If we don't assume that $M$ is totally real, we would encounter the fibers of the twistor fibration here. Take $X,Y$ of \Cref{eq:proof_EH_XY_lifted} and demand that it forms an involutive subalgebra with the isotropy. From demanding that it doesn't commutes into $\mathfrak{m}$ (except for $X$ and $Y$) we get some conditions on the projection of $X$ and $Y$ to $\mathfrak{h}$. Then, we demand that $([X, h_1] - \alpha Y)_\mathfrak{m}$ vanishes for some $\alpha$, and we find that this can never be. This shows that this case can also not occur and finishes the proof.
\end{proof}

\begin{table}[t]
        \centering
        \begin{tabular}{c:c|c:c}
            $\dim(\mathfrak{k}_0)$ & $\mathfrak{k}_0$ & $\dim(\mathfrak{k})$ & $\mathfrak{k}$ \\  \hline
             0 & $\set{0}$  & 2 & $\u(1) \oplus \u(1)$ \\
             1 & $\u(1)$  & 3 & $\su(2)$ \\
             2 & $\u(1) \oplus \u(1)$  & 4 & $\u(2)$, $\su(2) \oplus \u(1)$\\
             3 & $\su(2)$  & 5 & / \\
             4 & $\su(2) \oplus \u(1)$  & 6 & $\su(2)\oplus \su(2)$
        \end{tabular}
        \caption{The possible Lie sub-algebras $\mathfrak{k} = \Lie(K)$, $\mathfrak{k}_0 = \Lie(K_0)$ of $\sp(2)$ giving rise to an extrinsically homogeneous surface $K/K_0$ in $\CP^3$. }
        \label{table:possible_lie_algebras_EH_surfaces}
    \end{table}

\begin{lemma}\label{thm:characterisation_EH_not_S2}
    Suppose $M$ is an extrinsically homogeneous totally real surface of nearly Kähler $\CP^3$ that is not almost complex for Kähler $\CP^3$. Then, $M$ is totally real for Kähler $\CP^3$ and locally congruent to a flat torus. 
\end{lemma}
\begin{proof}
    By \Cref{thm:EH_only_S2_left}, we can assume that there is at least a one dimensional overlap between $TM$ and $\Dfour$ or $\Dtwo$. For an overlap with $\Dfour$, we can use the isotropy to assume at a point $p$ that $U_p= m_3$ and $V_p = a m_1 + \sum_{i=3}^6 \mu_i m_i$. Just as in the proof of \Cref{thm:EH_only_S2_left}, we can assume that $\mu_3 = 0$. The assumption that $M$ is totally real leads to the conclusion that $\mu_5 = 0$. Hence, after renaming, we have $V_p = a m_1 + \mu m_4 + \sqrt{2} \nu m_6$. The factor $\sqrt{2}$ is merely for future convenience. Notice that $M$ is now automatically also totally real for Kähler $\CP^3$. One can check with the exact same reasoning as before, that starting from $U_p = m_1$, one gets that $V_p =  m_3$. Hence, we can restrict to the case that there is a one dimensional overlap between $TM$ and $\Dfour$, and have the other as a special case therein. 

    Just as in the proof of \Cref{thm:EH_only_S2_left}, we can lift $U$ and $V$ to $X$ and $Y$ in $\Sp(2)$ which at $q$ (with $q \mapsto p$ under the map $\Sp(2) \to \CP^3 = \frac{\Sp(2)}{\Sp(1)\U(1)}$) have to be given by 
    \[ X_q = m_3 + \sum_{i=0}^3 a_i h_i \text{ and } Y_q = a m_1 + \mu m_4 + \nu m_6 + \sum_{i=0}^3 b_i h_i. \]
    Furthermore, we only have to consider the following cases as is argued in the proof of \Cref{thm:EH_only_S2_left}:
    \begin{enumerate}
        \item $X$ and $Y$ are commutative;
        \item $\{X,Y, [X,Y]\}$ span a $\SU(2)$ and $[X,Y] \in \mathfrak{h}$;
        \item There exist $Z_1, Z_2, Z_3, W \in \mathfrak{h}$ such that $\{Z_1,Z_2,Z_3, W\}$ generate $\su(2)\oplus\u(1)$ and $\{X,Y, Z_1,Z_2, Z_3, W\}$ generate $\su(2) \oplus \su(2)$.
    \end{enumerate}
    Notice how in the first case we obtain commuting orthonormal vector fields (after rescaling with constants) and so this example has to locally be a flat torus. 
    We now prove that there are no examples in the other two cases, and thus finish the proof.

    We assume that we are in the second case. By rescaling $X,Y, Z= [X,Y]_\mathfrak{h}$, we can assume that $[X,Y]=Z$, $[Y,Z]=X$ and $[Z,X]=Y$. Then, $g_q([Z,X]-Y, m_1) = 0$ implies $a=0$. Moreover, from $g_q([X,Y], m_1)= 0$ and $g_q([X,Y], m_2)= 0$, we find $\mu=\nu=0$. However, then $Y_\mathfrak{m}$ is zero, which is not allowed to happen. Hence, this case cannot occur. 

    Assuming we are in case 3, we conclude just as before that $\set{X,Y, h_0, h_1, h_2, h_3}$ needs to be involutive. In particular, we then need that $g_q([X, h_0] - x X -y Y, m_5) = 0$ for some $x,y$. However, $g_q([X, h_0] - x X -y Y, m_5) = -\sqrt{2}$, independent of $x,y$. Hence, also this case cannot occur. 
\end{proof}

\begin{proposition}\label{thm:EH_not_S2}
    Suppose $M$ is an extrinsically homogeneous totally real surface of nearly Kähler $\CP^3$. Then it is locally congruent to either a Kähler totally real torus of \Cref{eq:EH_tori_fam1,eq:EH_tori_fam2a,eq:EH_tori_fam2b,eq:EH_tori_fam3}, or to the Kähler almost complex sphere of \Cref{eq:ex_sphere_in_Lagrangian}. In particular, if it is minimal, it is locally congruent to \Cref{eq:EH_tori_fam1} with $h_1=\nu = 0$ or to \Cref{eq:EH_tori_fam2a} with $\nu = \pm \frac{1}{2 \sqrt{2}} \sqrt{1 + \sqrt{17}}$. Moreover, if it is minimal for Kähler $\CP^3$, it is either locally congruent to \Cref{eq:EH_tori_fam1} with $h_1=\nu = 0$, or to \Cref{eq:unique_TR_flat_minimal_KCP2}. 
\end{proposition}
\begin{proof}
    We start from the reasoning (and notation) of the proof of \Cref{thm:characterisation_EH_not_S2} to obtain $X_q = m_3 + \sum_{i=0}^3 a_i h_i$ and $Y_q = a m_1 + \mu m_4 + \nu m_6 + \sum_{i=0}^3 b_i h_i$ and that they have to commute with each other. From the fact that they commute, we find from $g([X,Y], m_1)$ and $g([X,Y], m_2)$ that $\mu=0$ and $\nu - \sqrt{2} a \, a_1=0$. Note that if $a=0$, that then $(Y_q)_\mathfrak{m}=0$, which is not possible. So, we can assume that $a \neq 0$ and by rescaling that $a=1$. Note that then $a_1 = \nu /\sqrt{2}$. By further solving the equations coming from $[X,Y]=0$, we find $a_3 = b_3 =0, b_4 = 1 - \nu^2 + \sqrt{2} \nu a_2$ and $b_2 = - \sqrt{2} \nu a_4-b_1$. Furthermore, from the remaining equation coming from commutativity of $X$ and $Y$, we can consider three cases. In the first case, we have $a_4 = 0, 2 a_2 = \nu$; for the second case, we have $a_4 = 0, a_2 = -\frac{1}{2\nu}$ and $\nu \neq 0$; and finally, for the third case, $b_1 = \frac{-4 a_4^2 \nu +2 a_2 \nu  (\nu -2 a_2)-2
   a_2 +\nu }{2 \sqrt{2} a_4}$ and $a_4 \neq 0$. Note that all cases lead to the same result on $\mathfrak{m}$. We then exponentiate the matrix $t X + s Y$ and let it act on $p_0 = (1,0,0,0)$. Reparametrising, we get the examples of \Cref{eq:EH_tori_fam1}; \Cref{eq:EH_tori_fam2a} and \Cref{eq:EH_tori_fam2b}; and \Cref{eq:EH_tori_fam3} for cases 1,2 and 3, respectively. We already noted exactly when these examples are minimal. 
\end{proof}

This result together with \Cref{thm:EH_not_S2} proves \Cref{thm:EH}.

%%%%%%%%%%%%%%%%%
\section{Codazzi totally real surfaces}
%%%%%%%%%%%%%%%%%

In this section, we classify totally real surfaces that are Codazzi, i.e. $\nabla h$ is totally symmetric, where $h$ is the second fundamental form. This class of submanifolds includes the parallel ones ($\nabla h = 0$) and totally geodesic submanifolds ($h=0$). We start with the following lemma.

\begin{lemma}\label{thm:CL_is_flat_Type2}
    Let $M$ be a totally real surface of nearly Kähler $\CP^3$. Then $M$ is of Type 2 and totally real for Kähler $\CP^3$ if and only if it is Codazzi. Moreover, in either case, it is a product and flat and never totally geodesic.
\end{lemma}
\begin{proof}
    Consider $M$ of Type 2 and totally real for Kähler $\CP^3$. From \Cref{thm:Kahler_TR_AC_type}, we know that $\cos\alpha$ has to vanish. By switching $U$ and $V$ if necessary, we can assume that $\alpha = \pi/2$. From \Cref{eq:derivatives_P_J0}, together with the Gauss and Codazzi equations, it follows that $M$ has to be Codazzi. 

    Conversely, suppose $M$ is Codazzi. Then, the Codazzi equation demands that $g( R(U,V)X, \xi)$ vanishes for $X \in \Span\set{U,V}$ and $\xi$ any normal to $M$. For Types 0 and 1, we immediately get a contradiction with this condition. For Type 2, the equations $g(R(U,V)U, JU) = 0$ and $g(R(U,V)UV, JV) = 0$ imply that $\cos\alpha=0$. As before, we can then assume $\alpha = \pi/2$ and that $M$ is totally real for Kähler $\CP^3$, by \Cref{thm:Kahler_TR_AC_type}.
    
    Using \Cref{eq:derivatives_P_J0}, we see that $d_1, d_2, m_1, m_3, m_4, k_2, k$ all have to vanish. Moreover, $m = 1/2$. Since $d_1=d_2=0$, we find that $U$ and $V$ are orthogonal coordinates, so that $M$ has to be flat. Since $m=1/2$, it is never totally geodesic. From the De Rham decomposition theorem, it further follows that $M$ has to be a product manifold. 
\end{proof}

%\subsection{Complete classification Codazzi totally real surfaces}

\begin{theorem}\label{thm:classification_CL_TR}
    A totally real surface in nearly Kähler $\CP^3$ is Codazzi if and only if it is a member of the family of examples of \Cref{eq:all_CL_TR} (after projection with $\hpi$).
\end{theorem}
\begin{proof}
    We already showed that all those examples are Codazzi. We just need to prove that they are the only ones.
    
    We start with the notations of \Cref{thm:CL_is_flat_Type2}. In particular, we are in Type 2 with $\alpha=\pi/2$. Introduce coordinates such that $\partial_u = \sqrt{2} U$ and $\partial_v = V$, and note that they are orthonormal with respect to $\FSm$.

    We now consider $\nabla h$ and impose it is totally symmetric. From $(\nabla h)(U,V,V)-(\nabla h)(V,U,V)= U(m_2) \, JV$, we find that $m_2$ can only depend on $v$. We change notation to $\kappa(v) = m_2$. 
    
    We compute that 
    $(\nabla h)(U,V,U)-(\nabla h)(V,U,U)= \l(-2 k_3 - \partial_v k_1 \r) J U+\l(2 k_1 - k_4 \kappa(v) -\partial_v k_3 \r) \cdot$ $G(U,V)+\l( k_3 \kappa(v) - \partial_v k_4 \r) J G(U,V)$.
    Hence, we get the system of PDE's
    \begin{equation}\label{eq:CL_TR_condeqlambda}
        \begin{cases}
            \partial_{v} k_1(u,v) &= -2 k_3(u,v)  \\
            \partial_{v} k_3(u,v) &= 2 k_1(u,v) - k_4(u,v) \kappa(v) \\
            \partial_{v} k_4(u,v) &= k_3(u,v) \kappa(v) 
        \end{cases}.
    \end{equation}
    Notice that this system can be treated as a system of ODE's in $v$, and letting the integration constants depend on $u$, so that existence and uniqueness of $\lambda_1, \lambda_2, \lambda_4$ is guaranteed with initial conditions. We take as initial conditions $k_1(u,0) = \frac{1}{\sqrt{2}} \kappa_1(u)$, $k_3(u,0) = - \frac{1}{\sqrt{2}} \kappa_3(u)$ and $k_4(u,0) =\frac{1}{\sqrt{2}} \kappa_2(u)$, for some functions $\kappa_1(u), \kappa_2(u), \kappa_3(u)$.

    Since $M$ is totally real for Kähler $\CP^3$ as well, we can take a horizontal lift to the round $\sphere{7} \subset \complexs^4$ of $M$. We will still work with the coordinates $u,v$ introduced above. Denote $\xi$ the unique horizontal lift of $\sqrt{2} G(U,V) = G(\partial_u, \partial_v)$. 
    We then have for the lift $F: I \times I \to \sphere{7}$ that the following equations need to be satisfied
    \begin{equation}\label{eq:CL_TR_2ndder}
        \begin{cases}
            F_{uu} = - F + \sqrt{2} k_1 i F_u + \sqrt{2} (k_4 - i k_3) i \xi \\
            F_{uv} = i \xi \\
            F_{vv} = - F - \kappa(v) i F_v
        \end{cases}.
    \end{equation}
    To finish the proof, we will study the system of equations \eqref{eq:CL_TR_2ndder}.

    Note, however, that up until now, we have lifted to the round $\sphere{7}$. In order to obtain an immersion for nearly Kähler $\CP^3$, we can still choose a quaternionic structure on $\complexs^4$ to identify it with $\quaternions^2$. Using the differential equations above, it is simple to show that the following definition gives rise to a (unique up to shift in $K$) parallel structure, and hence to a quaternionic structure on $\complexs^4$:
    \begin{equation}\label{eq:CL_TR_quaternionic_structure}
        j F = e^{i K(v)} F_v, \qquad
        j F_u = e^{i K(v)} i \xi,  \qquad
        j F_v = - e^{i K(v)} F,  \qquad
        j \xi = - i e^{i K(v)} F_u,
    \end{equation}
    where $K(v) = \int_0^v \kappa(t) \, \dd t$.

    Consider the first equation of \Cref{eq:CL_TR_2ndder} for $v=0$. As $K(0)=0$ and using \Cref{eq:CL_TR_quaternionic_structure}, it reduces to $F_{uu}(u,0) = -F(u,0) + \kappa_1(u) i F_u(u,0) + \kappa_2(u) j F_u(u,0) + \kappa_3(u) k F_u(u,0) $, with the initial conditions for $k_1, k_3, k_4$. We recognise this as an equation for a Legendre curve $\gamma: I \to \sphere{7}$. We thus have $F(u,0) = \gamma(u)$. Define $\quaternions(u)$ as 
    \[ \quaternions(u) = \Span\set{\gamma(u), i \gamma(u), j \gamma(u), k \gamma(u)}. \]
    We claim that for fixed $u$, $F(u,v)$ has to lie in $\quaternions(u)$. To show this, let $W(u)$ be any vector field with the property that along $\gamma$, it is orthogonal to $\gamma, i\gamma, j \gamma, k \gamma$, and that $\partial_v W = 0$. Since $F_v \in \Span\set{ j F, k F }$, we get the following system of equations (where the exact coefficients don't matter)
    \begin{align*}
        \partial_v \langle F, W \rangle &= \bullet \langle j F, W \rangle + \bullet \langle k F, W \rangle &
        \partial_v \langle iF, W \rangle& = \bullet \langle k F, W \rangle + \bullet \langle j F, W \rangle \\
        \partial_v \langle jF, W \rangle& = \bullet \langle F, W \rangle + \bullet \langle i F, W \rangle &
        \partial_v \langle kF, W \rangle& = \bullet \langle i F, W \rangle + \bullet \langle F, W \rangle.
    \end{align*}
    Note that for $v=0$, $\langle F(u,0), W\rangle = \langle \gamma(u), W\rangle = 0$, and similarly, $\langle i F(u,0), W\rangle = \langle j F(u,0), W\rangle = \langle k F(u,0), W\rangle = 0$. Hence, this system of differential equations in $v$ has the unique solution that for all $v$, it holds $\langle F, W\rangle = \langle i F, W\rangle = \langle j F, W\rangle = \langle k F, W\rangle = 0$, so that for fixed $u$ and all $v$, it holds that $F(u,v) \in \quaternions(u)$. 

    It then follows that the immersion has to take the form 
    \[ F(u,v) = f_1(v) \gamma(u) + f_2(v) i \gamma(u) + f_3(v) j \gamma(u)  + f_4(v) k \gamma(u).  \]
    We will write this more concisely as
    $F(u,v) = f(v) \gamma(u)$ and $f: I \to \quaternions$.
    From \Cref{eq:CL_TR_2ndder}, it follows that then
    %\[ 
    $f''(v) \gamma(u) = -f(v) \gamma(u) - \kappa(v) i f'(v) \gamma(u) \text{ or } f''(v) = -f(v) - \kappa(v) i f'(v)$. % \]
    With the identification of $\complexs^2$ and $\quaternions$, this is just the condition for a Legendre curve in $\sphere{3}$ with curvature $\kappa$, so we can assume that $f$ exists and is known. 

    We then consider the second condition of the system \eqref{eq:CL_TR_2ndder}. It is equivalent to
    %\begin{equation*}
        $f'(v) \gamma'(u) = i \xi = j e^{i K(v)} f(v) \gamma'(u)$
    %\end{equation*}
    where we used \Cref{eq:CL_TR_quaternionic_structure}. We claim that this equation is satisfied as well. Indeed, differentiate 
    \begin{equation}\label{eq:CL_TR_1stder}
        f'(v) = j e^{i K(v)} f(v).
    \end{equation}
    with respect to $v$, use that $K'(v) = \kappa(v)$, and observe that we simply get the Legendre equation for $f$ again. Hence, equations 2 and 3 of system \eqref{eq:CL_TR_2ndder} are satisfied by choosing the appropriate initial conditions for $f$. 
    
    We now consider the first equation of system \eqref{eq:CL_TR_2ndder} and claim it is also satisfied. Since $f$ lies in $\sphere{3}$, we have $\overline{f} f = 1$, so that it reduces to
    \[ \gamma''(u) = - \gamma(u) + \sqrt{2} k_1 \overline{f(v)} i f(v) \gamma'(u) + \sqrt{2} \; \overline{f(v)} ( k_4 - i k_3 ) f'(v) \gamma'(u). \]
    Using \Cref{eq:CL_TR_1stder}, this becomes
    \[ \gamma''(u) = - \gamma(u) + \sqrt{2} k_1 \overline{f(v)} i f(v) \gamma'(u) + \sqrt{2} \; \overline{f(v)} ( k_4 - i k_3 ) j e^{i K(v)} f(v) \gamma'(u). \]

    It takes a computation to expand $f$ in its components to see that 
    \[ \Re\l( \sqrt{2} k_1 \overline{f(v)} i f(v) + \sqrt{2} \; \overline{f(v)} ( k_4 - i k_3 ) j e^{i K(v)} f(v) \r) = 0 \]
    So that we can solve the following equation for $\tilde{\kappa}_1, \tilde{\kappa}_2, \tilde{\kappa}_3$:
    \[ \sqrt{2} k_1 \overline{f(v)} i f(v) + \sqrt{2} \; \overline{f(v)} ( k_4 - i k_3 ) j e^{i K(v)} f(v) = \tilde{\kappa}_1 i+\tilde{\kappa}_2 j+\tilde{\kappa}_3 k. \]
    Note that then the expression for $\gamma''$ takes the form of an equation for a Legendre curve in $\sphere{7}$. And indeed, differentiating the obtained expressions with respect to $v$, and using that $k_1, k_3, k_4$ are solutions to \Cref{eq:CL_TR_condeqlambda}, we find that $\tilde{\kappa}_1, \tilde{\kappa}_2, \tilde{\kappa}_3$ are independent of $v$. Hence, $\gamma$ has to be a solution to an equation for a Legendre curve. As $\tilde{\kappa}_1, \tilde{\kappa}_2, \tilde{\kappa}_3$ do not depend on $v$, and $\gamma$ was constructed to be a Legendre curve for $v= 0$, we find that the whole system \eqref{eq:CL_TR_2ndder} is indeed satisfied and $\tilde{\kappa}_i = \kappa_i$. In particular, this argument shows that the examples of \Cref{sec:examples_CL} are the unique Codazzi totally real surfaces of nearly Kähler $\CP^3$.
    %
    % 1) with frame: write down the equations and the conditions from the Codazzi equations, and take horizontal lift
    %
    % 2) for $v=0$, suppose $\lambda_2(u,0)=1$ $\lambda_2(u,0) = \lambda_4(u,0) = 0$ solve to find Legendre curve $\gamma$ in $\sphere{7}$
    %
    % 3) Show that for $u$ fixed, $F(u,v)$ has to lie in $\quaternions(u)$
    %
    % 4) solve $F_{vv}$ to find a Legendre curve in $\sphere{3}$
    %
    % 5) show that $i \xi = \pdv{u}\pdv{v} F$ is satisfied by the fact that $f$ is Legendre, and in doing so, find expressions for $f'$.
    %
    % 6) determine $\kappa_1, \kappa_2, \kappa_3$ so that $\gamma$ would be a Legendre curve
    %
    % 7) show that $\kappa_1, \kappa_2, \kappa_3$ are independent of $v$, finishing the argument.
\end{proof}

\begin{corollary}\label{thm:classification_parallel_1}
    Suppose $M$ is totally real for nearly Kähler $\CP^3$ and let $f: I \to \sphere{3} \subset \quaternions$ be an arc length parametrisation of a great circle, and $\gamma: I \to \sphere{3}$ a Legendre curve with constant curvature. Then $M$ is parallel, if and only if there exists an embedding $\iota : \sphere{3} \to \sphere{7} \subset \quaternions^2$ such that $M$ is of the form $\hpi \circ F$, where $F$ is given by
    $F: I \times I \mapsto \sphere{7}: f(v) (\iota\circ\gamma)(u)$.
    In particular, $M$ is minimal and parallel if and only if $\gamma$ is also a grand circle in $\sphere{3}$.
    
    Furthermore, if $M$ is parallel, it is extrinsically homogeneous. If it is minimal, it corresponds to \Cref{eq:EH_tori_fam1_minimal_parallel}, and otherwise it corresponds to one of the examples of \Cref{eq:EH_tori_fam3_parallel}.  
\end{corollary}
\begin{proof}
    We only need to show one direction. Suppose $M$ is parallel, then we find 
    \begin{align*}
        0 &= 2 (\nabla h)(U,V,U) = -k_3 J U + k1 G(U,V), &
        0 &= 2 g( (\nabla h)(U,V,V), G(U,V)) = \kappa,
    \end{align*}
    so that $k_1=k_3=\kappa = 0$. Then for all $s,t$ we have \[ g((\nabla h)(t U+s V, U,U), JG(U,V)) = t U(k_4) + s V(k_4),\] so that $k_4$ has to be constant.
    From the proof of \Cref{thm:classification_CL_TR}, we then have that $F$ has to be of the form
    $F: I \times I \to \sphere{7} \subset \quaternions: (u,v) \mapsto f(v) \gamma(u)$,
    with $f(v)$ a Legendre curve in $\sphere{3}$ with zero curvature, so a great circle, and $\gamma$ a Legendre curve with curvatures $\kappa_1 = \kappa_3 = 0$ and $\kappa_2= \sqrt{2} k_4$ constant. It then follows that $\gamma''(u) = - \gamma(u) + \kappa_2 j \gamma'(u)$, so that this is a Legendre curve in $\sphere{3} \subset \complexs^2$, with $j$ playing the role of complex unit. Hence, there exists an embedding of this $\complexs^2$ into $\quaternions^2$, so that $\iota$ is the induced embedding of $\sphere{3}$ in $\sphere{7}$.  

    From the above considerations, $M$ is minimal and parallel if and only if $k_4 = 0$, which is equivalent to $\gamma$ being a grand circle in $\sphere{3}$. 

    To link these parallel examples with the extrinsically homogeneous ones, note that all parameters are constants in the frame. With the same initial conditions, and uniqueness of the solutions of the resulting differential equations, we find that they are extrinsically homogeneous. 
\end{proof}

%%%%%%%%%%%%%%%%%
\section{Totally umbilical totally real surfaces}
%%%%%%%%%%%%%%%%%

\begin{theorem}\label{thm:classification_totally_umbilical}
    Let $M$ be a totally real surface of nearly Kähler $\CP^3$ that is also totally umbilical. Then, $M$ has to be congruent to example \eqref{eq:ex_sphere_in_Lagrangian}.
\end{theorem}
\begin{proof}
    From the assumption that $M$ is totally umbilic we find that $k_1=k_2=m_1=m_2=k=m=0$ and $m_3=k_3$, $m_4=k_4$. The only variables of the second fundamental form are hence $k_3$ and $k_4$. 

    We now consider the three types separately. If $M$ is of Type 1, it follows from \Cref{eq:derivatives_P_J0} applied to $(U,U,JV)$ and $(V,V,JV)$ that $1=0$, which is a contradiction.  
    If $M$ is of Type 2, we get from \Cref{eq:derivatives_P_J0} that $\alpha$ is constant. Moreover, applying the Ricci equation to $U,V,G(U,V),JG(U,V)$ and $U,V, U, J U$, we have that $\sin(2 \alpha) = 0$ and $k_3 = - \frac{1}{2} \cos\alpha$. From \Cref{eq:derivatives_P_J0}, it then follows that $\sin\alpha = 0$. \Cref{thm:Kahler_TR_AC_type} tells us that $M$ has to be almost complex for Kähler $\CP^3$ and then \Cref{thm:classification_AC_KCP3} tells us it is congruent to example \eqref{eq:ex_sphere_in_Lagrangian}.

    We finish the proof by showing that $M$ cannot be of Type 0. Suppose it is, then we consider three cases. First, we have $\lambda (1- \lambda^2) \neq 0$; second, $\lambda^2 = 1$; and third, $\lambda=0$.
    \begin{enumerate}
        \item[1) ]  In this case, the Ricci equation applied to $U,V, e_5, e_6$ and $U,V, U, JU$ gives that $\cos(\frac{\alpha}{2})=\sin(\frac{\alpha}{2})$, if $\lambda (\lambda^2-1) \neq 0$. This equation is never satisfied as $\cos \alpha \neq 0$. %So, this case does not have any totally umbilic submanifolds.
        \item[2) ] From the Codazzi equation applied to $U,V,U,JU$, we get an expression of $k_3$ in terms of $\alpha$. Then, \Cref{eq:derivatives_P_J0} implies that $U(\alpha) = \sqrt{2} \cos \theta$ and $V(\alpha) = \sqrt{2} \sin \theta$, so that $\alpha$ can never be constant. From the Codazzi equation applied to $U,V,U, e_5$, and the fact that $\alpha$ cannot be constant, we find that $\sin \theta = 0$. However, the Codazzi equation applied to $U,V,V, e_5$ then implies that $\alpha$ has to be constant, which is a contradiction. %Hence, also in this case, there are no totally umbilic examples.
        \item[3) ] From \Cref{eq:derivatives_P_J0}, it follows that $\theta = \pi/2$, $k_3=0$, $k_4 = d_2 (\tan\alpha -\sec\alpha )$, $2 d_1=\frac{\cos (\alpha )}{\sin (\alpha )-1}$, $U(\alpha) = 2 d_2 (\sec\alpha -\tan\alpha)$ and $V(\alpha) = -1$. Notice in particular that $\alpha$ cannot be constant. From the Codazzi equation, we get expressions of $U(d_2)$ and $V(d_2)$ in terms of $d_2$ and $\alpha$, where we use that $(\sec\alpha -\tan\alpha )$ cannot be identically zero, as $M$ cannot be totally geodesic. However, the compatibility equation $[U,V](d_2) = U(V(d_2)) - V(U(d_2))$
        implies that $\alpha$ has to be constant, which is a contradiction. \qedhere
    \end{enumerate}
\end{proof}

\pdfbookmark{Bibliography}{Bibliography}
\printbibliography

\end{document}